\documentclass[11pt,reqno]{amsart} % reqno right-aligns equation tags

\usepackage{
    amsfonts,
    amsmath,
    amssymb,
    amsthm
}

\usepackage{
    commath, % \abs
    eucal,
    mathrsfs, % \mathscr
    mathtools,
    nccmath, % \mfrac
    tikz,
    tikz-cd
}

\usetikzlibrary{
    decorations.pathreplacing,
    patterns.meta
}

\newcommand{\bbN}{\mathbb{N}}
\newcommand{\bbZ}{\mathbb{Z}}

\newcommand{\frakS}{\mathfrak{S}}
\newcommand{\frakG}{\mathfrak{G}}

\newcommand{\bfa}{\mathbf{a}}

\newcommand{\bfP}{\mathbf{P}}
\newcommand{\bfQ}{\mathbf{Q}}

\newcommand{\bfW}{\mathbf{W}}

\newcommand{\calH}{\mathcal{H}}
\newcommand{\calR}{\mathcal{R}}

\renewcommand{\phi}{\varphi}
\newcommand{\ol}[1]{\overline{#1}}

\newcommand{\sseq}{\subseteq}
\newcommand{\sm}{\smallsetminus}

\newcommand{\msert}{\overset{m}{\rightarrow}}

\DeclareMathOperator{\Fl}{\mathcal{F}\ell}

\DeclareMathOperator{\code}{code}
\DeclareMathOperator{\Grass}{gr}

\DeclareMathOperator{\Des}{Des}

\DeclareMathOperator{\wt}{wt}

\DeclareMathOperator{\PD}{\sf{PD}}
\DeclareMathOperator{\SPD}{\sf{SPD}}

\DeclareMathOperator{\RSSYT}{\sf{RSSYT}}

\DeclareMathOperator{\BM}{\sf{BM}}

\DeclareMathOperator{\RSK}{dRSK}

\DeclareMathOperator{\tab}{tab}
\DeclareMathOperator{\mat}{mat}

\DeclareMathOperator{\ins}{ins}
\DeclareMathOperator{\rec}{rec}

\DeclareMathOperator{\Rect}{Rect} % Rectification
\DeclareMathOperator{\coRect}{coRect} % coRectification

\newenvironment{shift}[1][(0,0)]{
    \begin{scope}[shift={#1}]
}{
    \end{scope}
}

\newtheorem{theorem}{Theorem}[section]
\newtheorem{proposition}[theorem]{Proposition}
\newtheorem{conjecture}[theorem]{Conjecture}
\newtheorem{lemma}[theorem]{Lemma}
\newtheorem{corollary}[theorem]{Corollary}

\theoremstyle{definition}
\newtheorem{definition}[theorem]{Definition}

\definecolor{pipe color}{RGB}{160,160,160}
\definecolor{x color 1}{RGB}{20,20,20}
\definecolor{x color 2}{RGB}{160,160,160}
\definecolor{y color 1}{RGB}{255,0,0}
\definecolor{y color 2}{RGB}{255,200,200}

\newcommand{\drawHorizontal}[3][color=pipe color]{
    \draw[#1,very thick]
    (#2-0.5,#3) to (#2+0.5,#3);
}
\newcommand{\drawVertical}[3][color=pipe color]{
    \draw[#1,very thick]
    (#2,#3-0.5) to (#2,#3+0.5);
}
\newcommand{\drawRelbow}[3][color=pipe color]{
    \draw[#1,very thick]
    (#2,#3-0.5) to (#2,#3-0.25) to[bend left] (#2+0.25,#3) to (#2+0.5,#3);
}
\newcommand{\drawJelbow}[3][color=pipe color]{
    \draw[#1,very thick]
    (#2-0.5,#3) to (#2-0.25,#3) to[bend right] (#2,#3+0.25) to (#2,#3+0.5);
}
\newcommand{\drawBump}[3][color=pipe color]{
    \drawRelbow[#1]{#2}{#3}
    \drawJelbow[#1]{#2}{#3}
}

\newcommand{\drawSquare}[3][black,opacity=0.5]{
    \draw[#1,thin]
    (#2-0.5,#3+0.5) rectangle (#2+0.5,#3-0.5);
}
\newcommand{\drawGraySquare}[3][gray,opacity=0.3]{
    \filldraw[#1,thin]
    (#2-0.5,#3+0.5) rectangle (#2+0.5,#3-0.5);
}

\newcommand{\drawXCross}[2]{
    \fill[color=x color 2]
    (#1-1/2,#2) -- (#1-1/4,#2+1/4) -- (#1+1/4,#2+1/4) -- (#1+1/2,#2) -- (#1+1/4,#2-1/4) -- (#1-1/4,#2-1/4);
    \drawHorizontal[x color 1]{#1}{#2}
    \drawVertical[x color 1]{#1}{#2}
}
\newcommand{\drawYCross}[2]{
    \fill[color=y color 2]
    (#1,#2-1/2) -- (#1+1/4,#2-1/4) -- (#1+1/4,#2+1/4) -- (#1,#2+1/2) -- (#1-1/4,#2+1/4) -- (#1-1/4,#2-1/4);
    \drawHorizontal[y color 1]{#1}{#2}
    \drawVertical[y color 1]{#1}{#2}
}
\newcommand{\drawXYCross}[2]{
    \fill[color=y color 2] (#1,#2) -- (#1+1/4,#2+1/4) -- (#1,#2+1/2) -- (#1-1/4,#2+1/4);
    \fill[color=y color 2] (#1,#2) -- (#1+1/4,#2-1/4) -- (#1,#2-1/2) -- (#1-1/4,#2-1/4);
    \fill[color=x color 2] (#1,#2) -- (#1+1/4,#2+1/4) -- (#1+1/2,#2) -- (#1+1/4,#2-1/4);
    \fill[color=x color 2] (#1,#2) -- (#1-1/4,#2+1/4) -- (#1-1/2,#2) -- (#1-1/4,#2-1/4);
    \drawVertical[y color 1]{#1}{#2}
    \drawHorizontal[x color 1]{#1}{#2}
}

\newcommand{\checkerradius}{0.38}
\newcommand{\drawXChecker}[2]{
    \fill[color=x color 2] (#1,#2) circle (\checkerradius);
    \fill[
        color=x color 1,
        pattern color=x color 1,
        pattern={Lines[distance=2pt]},
        thick
    ] (#1,#2) circle (\checkerradius);
    \draw[color=x color 1,thick] (#1,#2) circle (\checkerradius);
}
\newcommand{\drawYChecker}[2]{
    \fill[color=y color 2] (#1,#2) circle (\checkerradius);
    \fill[
        color=y color 1,
        pattern color=y color 1,
        pattern={Lines[distance=1.5pt,angle=90]},
    ] (#1,#2) circle (\checkerradius);
    \draw[color=y color 1, thick] (#1,#2) circle (\checkerradius);
}
\newcommand{\drawXYChecker}[2]{
    \drawXChecker{#1}{#2}
    \drawYChecker{#1-0.1}{#2+0.1}
}

\makeatletter

\newcount\PD@a
\newcount\PD@b
\newcount\PD@c

\newcount\PD@i
\newcount\PD@j

\newcount\PD@m
\newcount\PD@n

\newcount\PD@mode

\newcommand{\drawPDTile}[3]{
    \ifnum #3=0 \drawBump{#1}{#2} \fi
    \ifnum #3=1 \drawXCross{#1}{#2} \fi
    \ifnum #3=2 \drawYCross{#1}{#2} \fi
    \ifnum #3=3 \drawXYCross{#1}{#2}\fi
    \ifnum #3=4 \fi
    \ifnum #3=5 \fi
    \ifnum #3=6 \fi
    \ifnum #3=7 \fi
    \ifnum #3=8 \drawRelbow{#1}{#2} \fi
    \ifnum #3=9 \fi % blank tile (reserved)
}

\newcommand{\drawPD}[3]{
    \PD@i = 1
    \PD@j = 1
    \foreach \t in {#3}{
        \drawPDTile{\number\PD@j}{-\number\PD@i}{\t}
        \drawSquare{\number\PD@j}{-\number\PD@i}
        
        \ifnum\PD@j < \numexpr#2\relax
            \global\advance\PD@j by 1
        \else
            \global\PD@j = 1
            \global\advance\PD@i by 1
        \fi
    }
}

\newcommand{\drawRCGraphTile}[3]{
    \ifnum #3=0 \fi % blank tile (reserved)
    \ifnum #3=1 \drawXChecker{#1}{#2} \fi
    \ifnum #3=2 \drawYChecker{#1}{#2} \fi
    \ifnum #3=3 \drawXYChecker{#1}{#2} \fi
    \ifnum #3=4 \drawGraySquare{#1}{#2}\drawXChecker{#1}{#2} \fi
    \ifnum #3=5 \drawGraySquare{#1}{#2}\drawYChecker{#1}{#2} \fi
    \ifnum #3=6 \fi
    \ifnum #3=7 \fi
    \ifnum #3=8 \fi
    \ifnum #3=9 \drawGraySquare{#1}{#2} \fi
}

\newcommand{\drawRCGraph}[3]{
    \PD@i = 1
    \PD@j = 1
    \foreach \t in {#3}{
        \drawRCGraphTile{\number\PD@j}{-\number\PD@i}{\t}
        \drawSquare{\number\PD@j}{-\number\PD@i}
        
        \ifnum\PD@j < \numexpr#2\relax
            \global\advance\PD@j by 1
        \else
            \global\PD@j = 1
            \global\advance\PD@i by 1
        \fi
    }
}

\newcommand{\drawNorthLabels}[3]{
    \PD@j = 0
    \foreach \j in {#3}{
        \node at (#2+\PD@j,-#1) {\j};
        \global\advance\PD@j by 1
    }
}

\newcommand{\drawWestLabels}[3]{
    \PD@i = 0
    \foreach \i in {#3}{
        \node at (#2,-#1-\PD@i) {\i};
        \global\advance\PD@i by 1
    }
}

\makeatother

\newcommand{\bigtilescale}{.5}

\newcommand{\xcross}[1][scale=\tilescale,baseline=\tilebaseline]{\begin{tikzpicture}[#1]\drawSquare11\drawXCross11\end{tikzpicture}}
\newcommand{\ycross}[1][scale=\tilescale,baseline=\tilebaseline]{\begin{tikzpicture}[#1]\drawSquare11\drawYCross11\end{tikzpicture}}
\newcommand{\xycross}[1][scale=\tilescale,baseline=\tilebaseline]{\begin{tikzpicture}[#1]\drawSquare11\drawXYCross11\end{tikzpicture}}
\newcommand{\relbow}[1][scale=\tilescale,baseline=\tilebaseline]{\begin{tikzpicture}[#1]\drawSquare11\drawRelbow11\end{tikzpicture}}
\newcommand{\bump}[1][scale=\tilescale,baseline=\tilebaseline]{\begin{tikzpicture}[#1]\drawSquare11\drawBump11\end{tikzpicture}}

\newcommand{\bigxcross}{\xcross[scale=\bigtilescale]}
\newcommand{\bigycross}{\ycross[scale=\bigtilescale]}
\newcommand{\bigxycross}{\xycross[scale=\bigtilescale]}
\newcommand{\bigrelbow}{\relbow[scale=\bigtilescale]}

\newcommand{\bigbump}{\bump[scale=\bigtilescale]}

\title{Cauchy identities for Grothendieck polynomials and a dual RSK correspondence through pipe dreams}

\author{Hugh Dennin}
\thanks{The author was partially supported by NSF awards DMS-2231565 and DMS-1945212.}

\date{\today}

\begin{document}

\begin{abstract}
    The Cauchy identity gives a recipe for decomposing a double Grothendieck polynomial $\frakG^{(\beta)}_w(x;y)$ as a sum of products $\frakG^{(\beta)}_v(x)\frakG^{(\beta)}_u(y)$ of single Grothendieck polynomials.
    Combinatorially, this identity suggests the existence of a weight-preserving bijection between certain families of diagrams called pipe dreams.
    In this paper, we provide such a bijection using an algorithm called pipe dream rectification.
    In turn, rectification is built from a new class of flow operators which themselves exhibit a surprising symmetry.
    Finally, we examine other applications of rectification including an insertion algorithm on pipe dreams which recovers a variant of the dual RSK correspondence.
\end{abstract}

\maketitle

\setcounter{tocdepth}{1} % only show sections (and not subsections)
\tableofcontents

\section{Introduction}
\label{sect:intro}

Write $S_\infty$ for the set of permutations of $\bbN$ with finite support.
Each permutation $w\in S_\infty$ has an associated \textit{Schubert polynomial} $\frakS_w(x)\in \bbZ[x_i]_{i\in\bbN}$, first defined by Lascoux and Sch\"utzenberger in \cite{LascouxSchutzenberger82a}.
When $w\in S_n$, the Schubert polynomial $\frakS_w(x)$ represents the class of a corresponding Schubert variety $X_w$ inside the cohomology ring of the complete flag variety $\Fl(n)$.

\emph{Double Schubert polynomials} $\frakS_w(x;y) \in \bbZ[x_i;y_j]_{i,j\in\bbN}$ were introduced in \cite{Lascoux82} as a two-alphabet generalization of Schubert polynomials.
Double Schubert polynomials are related to (single) Schubert polynomials by the Cauchy identity.

\begin{theorem}[Cauchy identity for $\frakS_w$]
\label{thm:cauchy identity}
Let $w\in S_\infty$.
Then \[
    \frakS_w(x;y) =
    \sum_{\substack{u,v\in S_\infty \\ w \doteq u^{-1}v}}
    \frakS_v(x)\frakS_u(y),
\]
where the notation $w \doteq u^{-1}v$ indicates that both $w = u^{-1}v$ and $\ell(w) = \ell(u) + \ell(v)$.
\end{theorem}

In the $K$-theory of $\Fl(n)$, the class of the structure sheaf of a Schubert variety is instead represented by a \textit{Grothendieck polynomial} $\frakG_w(x)\in \bbZ[x_i]_{i\in\bbN}$ \cite{LascouxSchutzenberger82b,LascouxSchutzenberger83}.
The Grothendieck polynomials we will consider $\frakG_w^{(\beta)}(x)$ are a variant introduced in \cite{FominKirillov94} with an additional homogenizing $\beta$-parameter.
We can recover $\frakS_w(x)$ and $\frakG_w(x)$ from $\frakG_w^{(\beta)}(x)$ by specializing $\beta = 0$ and $\beta = -1$, respectively.

As with Schubert polynomials, we have a notion of \textit{double Grothendieck polynomials} $\frakG_w^{(\beta)}(x;y) \in \bbZ[\beta][x_i;y_j]_{i,j\in\bbN}$.
Grothendieck polynomials also satisfy a Cauchy identity.

\begin{theorem}[Cauchy identity for $\frakG_w^{(\beta)}$ \cite{FominKirillov94}]
\label{thm:K-cauchy identity}
Let $w\in S_\infty$.
Then \[
    \frakG_w^{(\beta)}(x;y) =
    \sum_{\substack{u,v\in S_\infty \\ w = u^{-1}*v}} 
    \beta^{\ell(u)+\ell(v)-\ell(w)}\frakG_v^{(\beta)}(x)\frakG_u^{(\beta)}(y),
\]
where $u^{-1}*v$ denotes the product of $u^{-1}$ and $v$ in the Demazure algebra.
\end{theorem}

Schubert and Grothendieck polynomials can be viewed as generating functions for certain collections of wiring diagrams called \emph{pipe dreams}; see \cite{BergeronBilley93,FominKirillov96,FominKirillov94,KnutsonMiller05}.
Their doubled variants have a similar interpretation in terms of \textit{super pipe dreams}---pipe dreams with a coloring of their crossings using black and red.

Equipped with this viewpoint, the Cauchy identities (Theorems \ref{thm:cauchy identity} and \ref{thm:K-cauchy identity}) become combinatorial identities of pipe dreams, so a natural question is to ask for a weight preserving bijection that realizes them.
In practice, we would like an algorithm that takes a super pipe dream for $w$ as input and produces two pipe dreams corresponding to a decomposition of $w$ into a pair $u$ and $v$.

The main goal of this paper is to provide such an algorithm, which we call pipe dream \textit{rectification}.
Rectification will be constructed from a new class of operators on super pipe dreams that \textit{flow} crosses between columns and rows.
We will then examine some further applications of rectification, including an insertion algorithm on pipe dreams which recovers a variant of the \textit{dual Robinson--Schensted--Knuth} correspondence.

This paper is organized as follows:
in Sections \ref{sect:combinatorics} and \ref{sect:pipe dreams}, we recall the necessary background on permutations and pipe dreams and give our definition of a super pipe dream.
In Section \ref{sect:flow ops}, we define the flow operators $Y_j^+$ and $Y^+$ which advance red crosses in a super pipe dream to the right by column, as well as their counterparts $X_i^+$ and $X^+$ acting on black crosses.
We observe a surprising symmetry between the flow operators $X^+$ and $Y^+$ (Proposition \ref{prop:flow symmetry}) in Section \ref{sect:flow symmetry}.

Rectification is defined in Section \ref{sect:cauchy rectification}, where it is used to give a new combinatorial proof of the Cauchy identities for Schubert and Grothendieck polynomials.
Further applications of flowing and rectification are explored in Section \ref{sect:applications}.
These include bijective proofs of a derivative formula for Grothendieck polynomials first observed in \cite{PechenikSpeyerWeigandt21} (Section \ref{subsect:nabla identities}) and a restricted descent Pieri rule (Section \ref{subsect:pieri rule}), the latter of which will lead to a new insertion algorithm on pipe dreams.
We also derive an identity for Stanley symmetric functions (Section \ref{subsect:Stanley symmetric}) and a recurrence for double Grothendieck polynomials using Bergeron--Sottile operators generalizing that of \cite{NadeauSpinkTewari24} (Section \ref{subsect:R_k recurrence}).

Finally, in Section \ref{sect:RSK} we obtain our variant of dual RSK by applying the insertion algorithm  rectification to reduced pipe dreams for Grassmannian permutations, which are known to be in 1-to-1 correspondence with semi-standard Young tableaux.
We will see that the ``insertion tableau'' we obtain through our variant of dual RSK agrees with the insertion tableau of the usual dual RSK correspondence (Proposition \ref{prop:pipe dream dual RSK}).

\section{Permutations}
\label{sect:combinatorics}

We use the notation $\bbN := \{1,2,\dots\}$ and $[n] := \{1,\dots,n\}$.

Write $S_n$ for the group of permutations on $\bbN$ with \textit{support} $\{i \mid w(i)\neq i\}$ contained in $[n]$, and set $S_\infty := \bigcup_{n=1}^\infty S_n$.
We will often write permutations in $S_\infty$ using one-line notation.
For example, $1423$ denotes the permutation $w\in S_4$ with $1\mapsto 1$ and $2\mapsto 4\mapsto 3\mapsto 2$.

For $i\in \bbN$, write $s_i\in S_\bbN$ for the simple transposition that swaps $i$ and $i+1$.
The simple transpositions $s_1,\dots,s_{n-1}$ generate $S_n$ under the relations \begin{align*}
    s_i^2 &= 1, \\
    s_is_j &= s_js_i
    \quad\text{(for $\abs{i-j}>1$)},
    % \tag{commutation relation}
    \\
    s_is_{i+1}s_i &= s_{i+1}s_is_{i+1}.
    % \tag{braid relation}
\end{align*}

A \textit{word} for $w\in S_\infty$ is a sequence $\mathbf a := (a_1,\dots,a_r)$ of positive integers such that $w = s_{a_1}\cdots s_{a_r}$.
The minimum length $\ell(w)$ among words for $w$ is called the \textit{length} of $w$, and words for $w$ attaining this length are called \textit{reduced}.

The \textit{code} of $w$ is the sequence $\code(w) = (c_1,c_2,\dots)\in \bbN^\infty$ determined by $c_i = \abs{\{j> i \mid w(j) < w(i)\}}$.
This is a refinement of the length of $w$ in the sense that $\ell(w) = \sum_i c_i$.
Codes give a bijection between $S_\infty$ and the subset of $\bbN^\infty$ of sequences with finitely many nonzero terms.

The \textit{descent set} of $w$ is the collection of indices $\Des(w) = \{i \mid w(i) > w(i+1)\}$.
Notice that $i\in \Des(w)$ if and only if $\ell(ws_i) = \ell(w) - 1$.

The \textit{Demazure algebra} (or \textit{$0$-Hecke algebra}) is the algebra $H$ generated by elements $\{e_i \mid i\in\bbN\}$ subject to the relations \begin{align*}
    e_i^2 &= e_i, \\
    e_ie_j &= e_je_i
    \quad\text{ (for $\abs{i-j}>1$)},
    \\
    e_ie_{i+1}e_i
    &= e_{i+1}e_ie_{i+1}.
\end{align*}
There is an isomorphism of the underlying abelian group of the group ring $\bbZ[S_\infty]$ with that of $H$ given by sending the basis element $w\in S_\infty$ to the element $\partial(\mathbf a) := e_{a_1}\cdots e_{a_{\ell}}\in H$, where $\mathbf a = (a_1,\dots,a_{\ell})$ is any reduced word for $w$.
By abuse of notation, we will identify $w$ with the element $\partial(\mathbf a)$.
If $u,v\in S_\infty$, then we write $u*v$ to distinguish the product of $u$ and $v$ in $H$ from their usual product $uv$ in $\bbZ[S_\infty]$.

\section{Pipe dreams}
\label{sect:pipe dreams}

\begin{figure}
\begin{tikzpicture}[scale=0.45]
    \begin{shift}
    \draw[very thick] (2.5,-7.5) -- (2.5,-0.5);
    \draw[very thick] (0.5,-2.5) -- (7.5,-2.5);
    \drawRCGraph{7}{7}{
        9,9,9,9,0,0,0,
        9,9,9,0,0,1,0,
        9,9,1,1,0,0,0,
        9,0,1,0,0,0,0,
        0,0,0,0,0,0,0,
        0,0,0,0,1,0,0,
        1,0,0,0,0,0,0
    }
    \drawNorthLabels{0}{1}{-1,0,1,2,3,4,5}
    \drawWestLabels {1}{0}{-1,0,1,2,3,4,5}

    \draw[<->] (8.5,-4) -- (11.5,-4);
    \end{shift}
    
    \begin{shift}[(12,0)]
    \draw[very thick] (2.5,-7.5) -- (2.5,-0.5);
    \draw[very thick] (0.5,-2.5) -- (7.5,-2.5);
    \drawPD{7}{7}{
        9,9,9,8,0,0,0,
        9,9,8,0,0,1,0,
        9,8,1,1,0,0,0,
        8,0,1,0,0,0,0,
        0,0,0,0,0,0,0,
        0,0,0,0,1,0,0,
        1,0,0,0,0,0,0
    }
    \end{shift}
\end{tikzpicture}

\caption{
    A reduced pipe dream $P$ with permutation $w = \partial(3,2,1,2,6,3) = 4231576$.
}
\label{fig:pipe dream}
\end{figure}
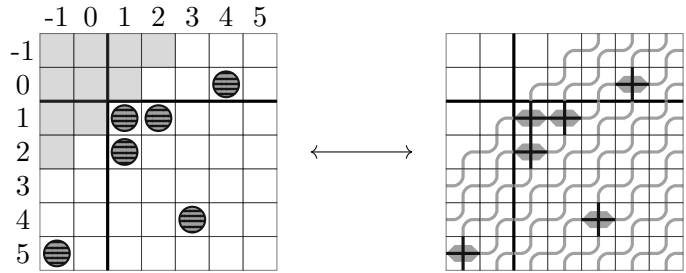

\subsection{Pipe dreams}
\label{subsect:pipe dreams}

Visualize $\bbZ^2$ as an infinite playing board of square tiles.
A position $(i,j)\in \bbZ^2$ will refer to the tile in row $i$ and column $j$, with row indices increasing downwards and column indices increasing rightwards.

\begin{definition}[\cite{FominKirillov96,BergeronBilley93,KnutsonMiller05}]
\label{def:pipe dreams}
A \textit{(single) pipe dream} $P$ is any finite subset of the half-space
\[
    \calH := \{(i,j)\in \bbZ^2 \mid i+j-1 \geq 1\}.
\]
We regard $P$ as a placement of finitely many black checkers in the region $\calH$, with each position being occupied by at most one checker.
\end{definition}

Write $\PD$ for the set of all pipe dreams.
The terminology for a pipe dream $P$ comes from the more typical visualization of $P$ as a tiling of $\calH_0 := \{(i,j)\in \bbZ^2 \mid i+j-1\geq 0\}$ using \textit{cross}, \textit{bump}, and \textit{elbow} tiles \[
    \bigxcross\qquad\bigbump\qquad\bigrelbow
\] whose locations are given by $P$, $\calH\sm P$, and $\calH_0\sm \calH$, respectively.
Such a tiling depicts a configuration of countably many pipes winding in the plane.
Our visualization using checkers is more in line with the original \textit{RC-graph} terminology for $P$.
See Figure \ref{fig:pipe dream} for an example of both visualizations.

Each pipe dream $P$ has a \textit{word} $\bfa(P) := (a_1,\dots,a_p)$ obtained by reading off the values $i+j-1$ of each $(i,j)\in P$ from right-to-left, beginning with the top row and moving down
\cite{BergeronBilley93,FominKirillov96,FominKirillov94}.
The \textit{permutation} of $P$ is the permutation $\partial(P) := \partial(\bfa(P))$ (see Figure \ref{fig:pipe dream}).

Say that $P\in \PD$ is \emph{reduced} if $\bfa(P)$ is a reduced word.
Write $\PD_0$ for the set of all reduced pipe dreams.
If $P\in \PD(w)$, then $P$ is reduced if and only if $\abs{P} = \ell(w)$.

For each $w\in S_\infty$, write $\PD(w)$ for the set of pipe dreams with permutation $w$.
There is a natural bijection $\PD(w)\to \PD(w^{-1})$ given by sending a pipe dream $P$ to its \textit{transpose} \[
    P^t := \{(j,i) \mid (i,j)\in P\}.
\]

A pipe dream $P\in \PD$ is \textit{ordinary} if $P\sseq \bbN\times\bbN$ and we write $\PD^+$ for the collection of these pipe dreams.
Our conventions differ from the standard ones in that what we call ordinary pipe dreams are just pipe dreams elsewhere (e.g. \cite{BergeronBilley93,KnutsonMiller05}).

Throughout this paper, we will freely allow decoration combos on sets of pipe dreams.
For instance, $\PD_0^+(w)$ is understood to mean the family $\PD_0\cap \PD^+\cap \PD(w)$ of reduced, ordinary pipe dreams with permutation $w$.

\subsection{Ladder moves and chute moves}
\label{subsect:ladder moves}

\begin{figure}
\begin{tikzpicture}[scale=0.45]
    \draw[decorate,decoration={brace,amplitude=4pt},xshift=6pt,yshift=0]
    (0,-5.5) -- (0,-0.5) node [midway,left,xshift=-2pt] {$k$};

    % PD #1
    \begin{scope}[shift={(0,0)}]
    \drawRCGraph{2}{2}{0,0,1,1}
    \node at (1,-2.75) {$\vdots$};
    \node at (2,-2.75) {$\vdots$};
    \begin{scope}[shift={(0,-3)}]
    \drawRCGraph{2}{2}{1,1,1,0}
    \end{scope}
    \node at (1.5,-6.5) {$P$};
    % \draw[->] (3,-3) -- (4,-3);
    \end{scope}
    
    % PD #2
    \begin{scope}[shift={(4,0)}]
    \drawRCGraph{2}{2}{0,1,1,1}
    \node at (1,-2.75) {$\vdots$};
    \node at (2,-2.75) {$\vdots$};
    \begin{scope}[shift={(0,-3)}]
    \drawRCGraph{2}{2}{1,1,0,0}
    \end{scope}
    \node at (1.5,-6.5) {$Q$};
    \end{scope}

    % PD #3
    \begin{scope}[shift={(8,0)}]
    \drawRCGraph{2}{2}{0,1,1,1}
    \node at (1,-2.75) {$\vdots$};
    \node at (2,-2.75) {$\vdots$};
    \begin{scope}[shift={(0,-3)}]
    \drawRCGraph{2}{2}{1,1,1,0}
    \end{scope}
    \node at (1.5,-6.5) {$R$};
    \end{scope}
\end{tikzpicture}
\caption{
    Illustration of a ladder move.
}
\label{fig:ladder move}
\end{figure}
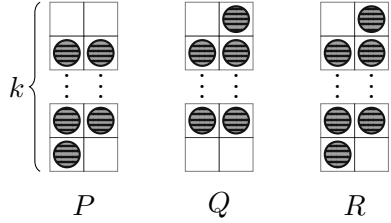

We set some terminology:
for $k\geq 1$, a \textit{$k$-ladder} will refer to any $k$-by-$2$ rectangle $L$ situated inside $\bbZ^2$.
Similarly, a \textit{$k$-chute} is any $2$-by-$k$ rectangle $C$ in $\bbZ^2$.
Ladders/chutes where $k\geq 2$ are said to be \textit{big}.

Suppose that $P$, $Q$, and $R$ are pipe dreams that are the same everywhere, except for in big ladder $L$ where the local configurations are shown in Figure \ref{fig:ladder move}.
It can be checked in both cases that $\partial(P) = \partial(Q) = \partial(R)$.
In this case, say that $P$, $Q$, and $R$ are obtained from one another by a \textit{ladder move} \cite{BergeronBilley93,Tyurin18}.

The ladder move between $P$ and $Q$ is special since in this case $\abs{P} = \abs{Q}$.
In particular, $P$ is reduced if and only if $Q$ is reduced, so we refer to this kind of ladder move as a \textit{reduced ladder move}.
Visualizing $P$ and $Q$ as wiring diagrams, notice that two border edges of $L$ are connected by a pipe in $P$ if and only if they are connected by a pipe in $Q$.

There is dually a notion of a \textit{chute move}, which is just the transpose of a ladder move.

\subsection{Super pipe dreams}
\label{subsect:super pipe dreams}

\begin{definition}
\label{def:super pipe dreams}
A \textit{super pipe dream} is a pair of pipe dreams $\bfP = (P_x,P_y)$,\footnote{
    The terminology comes from thinking of $\bfP$ as a $\bbZ/2\bbZ$-graded pipe dream.
} which we visualize as a placement of finitely many black and {\color{y color 1}red} checkers in $\calH$ whose positions are given, respectively, by $P_x$ and $P_y$.
We do not require $P_x$ and $P_y$ to be disjoint, so a position may contain both a black checker and a {\color{y color 1}red} checker (but no more than one of each).
\end{definition}

Write $\SPD$ for the set of all super pipe dreams.
We regard $\PD$ as a subset of $\SPD$ via the embedding $P\mapsto (P,\varnothing)$.
If we wish to visualize a super pipe dream $\bfP$ as a wiring diagram in $\calH_0$, we can do so using the tiles \[
    \bigxcross\qquad
    \bigycross\qquad
    \bigxycross\qquad
    \bigbump\qquad
    \bigrelbow.
\]
See Figure \ref{fig:super pipe dream} for an example.

A super pipe dream $\bfP = (P_x,P_y)$ is \textit{ordinary} if its \textit{underlying pipe dream} $P = P_x\cup P_y$ is ordinary.
Likewise, the \textit{word} and \textit{permutation} of $\bfP$ are the word and permutation of $P$.
Say that $\bfP$ is \textit{reduced} if $P$ is reduced and $P_x\cap P_y = \varnothing$.
Equivalently, $\bfP \in \SPD(w)$ is reduced if and only if $\abs{P_x} + \abs{P_y} = \ell(w)$.
Use the same ornamentations (e.g. $\SPD^+, \SPD_0, \SPD(w)$) as with single pipe dreams to denote the corresponding sets of super pipe dreams.

We end this section with three involutions on super pipe dreams: let $\bfP\in \SPD$.
\begin{itemize}
\item
The \textit{complement} of $\bfP$ is $\ol{\bfP} := (P_y,P_x)$.
\item
The \textit{transpose} of $\bfP$ is $\bfP^t := (P_x^t,P_y^t)$.
\item
The \textit{adjoint} of $\bfP$ is $\bfP^\dagger := \ol{\bfP^t} = \ol{\bfP}^t = (P_y^t,P_x^t)$.
\end{itemize}
Complementation permutes each $\SPD(w)$ while transposition and adjunction give bijections $\SPD(w)\to \SPD(w^{-1})$.

\begin{figure}
\begin{tikzpicture}[scale=0.45]
    \begin{shift}
    \draw[very thick] (2.5,-7.5) -- (2.5,-0.5);
    \draw[very thick] (0.5,-2.5) -- (7.5,-2.5);
    \drawRCGraph{7}{7}{
        9,9,9,9,0,0,3,
        9,9,9,1,0,0,0,
        9,9,0,0,3,0,0,
        9,0,0,1,2,0,0,
        2,2,0,1,0,0,0,
        0,0,0,0,0,0,0,
        0,1,0,0,0,2,0
    }
    \drawNorthLabels{0}{1}{-1,0,1,2,3,4,5}
    \drawWestLabels{1}{0}{-1,0,1,2,3,4,5}
    \end{shift}
    
    \draw[<->] (8.5,-4) -- (11.5,-4);
    
    \begin{shift}[(12,0)]
    \draw[very thick] (2.5,-7.5) -- (2.5,-0.5);
    \draw[very thick] (0.5,-2.5) -- (7.5,-2.5);
    \drawPD{7}{7}{
        9,9,9,8,0,0,3,
        9,9,8,1,0,0,0,
        9,8,0,0,3,0,0,
        8,0,0,1,2,0,0,
        2,2,0,1,0,0,0,
        0,0,0,0,0,0,0,
        0,1,0,0,0,2,0
    }
    \end{shift}
\end{tikzpicture}

\caption{
    A super pipe dream $\bfP$.
    % with $x$-crosses 
    % $\{(-1,5),(0,2),(1,3),(2,2),(3,2),(5,0)\}$
    % and $y$-crosses
    % $\{(-1,6),(-1,5),(1,3),(2,3),(3,0),(3,-1)\}$.
}
\label{fig:super pipe dream}
\end{figure}

\subsection{Schubert and Grothendieck polynomials}
\label{subsect:Schubert polys}

The \textit{weight} of a super pipe dream $\bfP\in \SPD(w)$ is the monomial in $\bbZ[\beta][x;y] := \bbZ[\beta][x_i;y_j]_{i,j\in\bbZ}$ given by \[
    \wt(\bfP) := \beta^{\abs{P_x}+\abs{P_y}-\ell(w)}
    \bigg(\prod_{(i,j)\in P_x} x_i\bigg)
    \bigg(\prod_{(i,j)\in P_y} y_j\bigg).
\]
Note that if $P\in \PD(w)$, then (identifying $P$ with $(P,\varnothing)\in \SPD(w)$) \[
    \wt(P) = \beta^{\abs{P}-\ell(w)}\prod_{(i,j)\in P} x_i.
\]

\begin{definition}[\cite{BergeronBilley93,FominKirillov96,FominKirillov94}]
Let $w\in S_\infty$.
The
\textit{Schubert polynomial} and
\textit{($\beta$-)Grothendieck polynomial} of $w$,
as well as their doubled variants,\footnote{
    Our definition of a double Schubert polynomial is nonstandard and is typically what one would call $\frakS_w(x;-y)$.
}
are defined to be the polynomials
\begin{align*}
    \frakS_w(x)
    &:= \sum_{P\in \PD^+_0(w)} \wt(P),
    &
    \frakG_w^{(\beta)}(x)
    &:= \sum_{P\in \PD^+(w)} \wt(P),
    \\
    \frakS_w(x;y)
    &:= \sum_{\bfP\in \SPD_0^+(w)} \wt(\bfP),
    &
    \frakG_w^{(\beta)}(x;y)
    &:= \sum_{\bfP\in \SPD^+(w)} \wt(\bfP).
\end{align*}
\end{definition}

Notice that $\frakS_w(x) = \frakG_w^{(0)}(x)$ and $\frakS_w(x;y) = \frakG_w^{(0)}(x;y)$, as well as $\frakS_w(x) = \frakS_w(x;0)$ and $\frakG_w^{(\beta)}(x) = \frakG_w^{(\beta)}(x;0)$.
If we view $\beta$ as being in degree $-1$ (and each $x_i$ and $y_j$ as being in degree $1$), then each of the above polynomials is homogeneous of degree $\ell(w)$.

Let $(-)^\dagger$ denote the ring map on $\bbZ[\beta][x;y]$ given by swapping the sets of $x$ and $y$ variables (i.e. $x_i\leftrightarrow y_i$ for $i\in\bbZ$).
Then $\wt(\bfP^\dagger) = \wt(\bfP)^\dagger$ for all $\bfP\in \SPD$, hence we get the identity \[
    \frakG_{w^{-1}}^{(\beta)}(x;y)
    = \sum_{\bfP\in \PD^+(w^{-1})} \wt(\bfP)
    = \sum_{\bfQ\in \PD^+(w)} \wt(\bfQ)^\dagger
    = \frakG_w^{(\beta)}(y;x)
\]

\section{Flow operators on super pipe dreams}
\label{sect:flow ops}

\begin{figure}
\begin{tikzpicture}[scale=0.45]
    % DIAGRAM 1
    \begin{shift}
    \draw[very thick] (0.5,-6.5) -- (0.5,-0.5);
    \draw[very thick] (0.5,-1.5) -- (6.5,-1.5);
    \drawRCGraph{6}{6}{
        9,0,0,0,0,0,
        1,0,0,2,0,0,
        3,2,1,2,3,0,
        2,0,3,0,0,0,
        1,3,0,0,1,0,
        2,2,0,1,0,0
    }
    \drawNorthLabels{0}{1}{1,2,3,4,5,6}
    \drawWestLabels{1}{0}{0,1,2,3,4,5}
    \node at (3.5,-7.5) {$\bfP = Y^+_{\geq 6}\bfP$};
    \draw[->] (7,-4) -- (9,-4)
    node[midway,above] {$Y^+_5$};
    \end{shift}

    % DIAGRAM 2
    \begin{shift}[(9,0)]
    \draw[very thick] (0.5,-6.5) -- (0.5,-0.5);
    \draw[very thick] (0.5,-1.5) -- (6.5,-1.5);
    \drawRCGraph{6}{6}{
        9,0,0,0,0,0,
        1,0,0,2,0,2,
        3,2,1,2,1,0,
        2,0,3,0,0,0,
        1,3,0,0,1,0,
        2,2,0,1,0,0
    }
    \node at (3.5,-7.5) {$Y^+_{\geq 5}\bfP$};
    \draw[->] (7,-4) -- (9,-4)
    node[midway,above] {$Y^+_4$};
    \end{shift}

    % DIAGRAM 3
    \begin{shift}[(18,0)]
    \draw[very thick] (0.5,-6.5) -- (0.5,-0.5);
    \draw[very thick] (0.5,-1.5) -- (6.5,-1.5);
    \drawRCGraph{6}{6}{
        9,0,0,0,2,0,
        1,0,0,0,0,2,
        3,2,1,1,2,0,
        2,0,3,0,0,0,
        1,3,0,0,1,0,
        2,2,0,1,0,0
    }
    \node at (3.5,-7.5) {$Y^+_{\geq 4}\bfP$};
    \draw[->] (7,-4) -- (8,-4);
    \end{shift}

    % DIAGRAM 4
    \begin{shift}[(0,-9)]
    \draw[->] (-2,-4) -- (0,-4)
    node[midway,above] {$Y^+_3$};
    \draw[very thick] (0.5,-6.5) -- (0.5,-0.5);
    \draw[very thick] (0.5,-1.5) -- (6.5,-1.5);
    \drawRCGraph{6}{6}{
        9,0,0,0,2,0,
        1,0,0,2,0,2,
        3,2,1,1,2,0,
        2,0,1,0,0,0,
        1,3,0,0,1,0,
        2,2,0,1,0,0
    }
    \node at (3.5,-7.5) {$Y^+_{\geq 3}\bfP$};
    \draw[->] (7,-4) -- (9,-4)
    node[midway,above] {$Y^+_2$};
    \end{shift}

    % DIAGRAM 5
    \begin{shift}[(9,-9)]
    \draw[very thick] (0.5,-6.5) -- (0.5,-0.5);
    \draw[very thick] (0.5,-1.5) -- (6.5,-1.5);
    \drawRCGraph{6}{6}{
        9,0,0,0,2,0,
        1,0,0,2,0,2,
        3,1,2,1,2,0,
        2,0,3,0,0,0,
        1,0,3,0,1,0,
        2,0,0,1,0,0
    }
    \node at (3.5,-7.5) {$Y^+_{\geq 2}\bfP$};
    \draw[->] (7,-4) -- (9,-4)
    node[midway,above] {$Y^+_1$};
    \end{shift}

    % DIAGRAM 6
    \begin{shift}[(18,-9)]
    \draw[very thick] (0.5,-6.5) -- (0.5,-0.5);
    \draw[very thick] (0.5,-1.5) -- (6.5,-1.5);
    \drawRCGraph{6}{6}{
        9,2,0,0,2,0,
        0,1,0,2,0,2,
        1,3,2,1,2,0,
        0,2,3,0,0,0,
        0,1,3,0,1,0,
        0,0,0,1,0,0
    }
    \node at (3.5,-7.5) {$Y^+_{\geq 1}\bfP = Y^+\bfP$};
    \end{shift}
\end{tikzpicture}

\caption{
    Computation of $Y^+\bfP$ for a super pipe dream $\bfP\in \SPD^+(w)$ for $w = 273561498$.
}
\label{fig:flow example}
\end{figure}
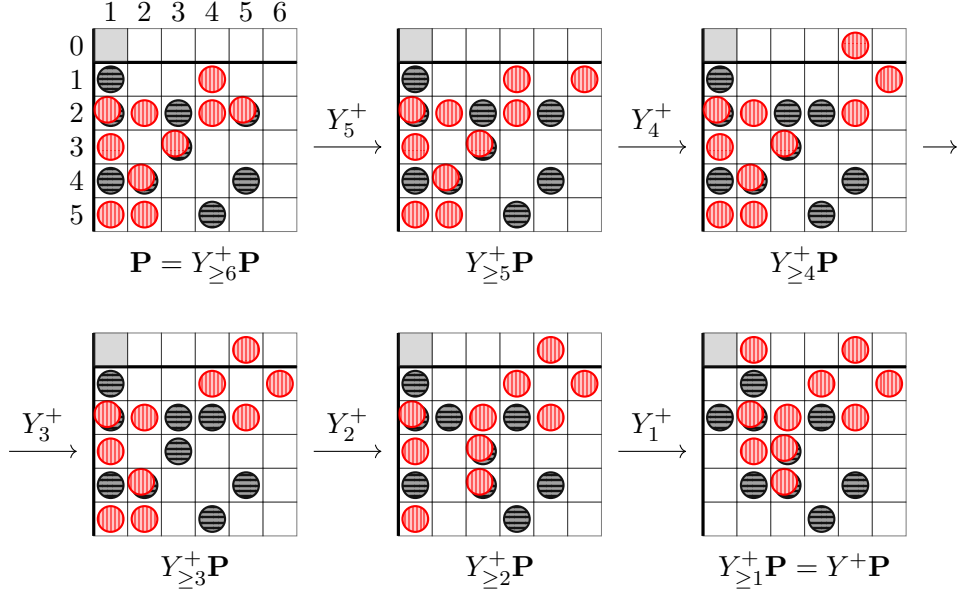

In this section, we will define a collection of operators on super pipe dreams which will ``flow'' red checkers between columns while preserving the row distribution of black checkers, and vice versa.
These operators will be the fundamental building block which allows us to separate a super pipe dream into two single pipe dreams.

\subsection{Flowing red checkers}
\label{subsect:flow definition}

We begin by defining, for each $j\in \bbZ$, an operator $Y^+_j$ that takes a super pipe dream $\bfP$ with no red checkers in column $j+1$ and modifies it by ``flowing'' the red checkers in column $j$ to into column $j+1$.
More precisely,

\begin{proposition}
    Let $w\in S_\infty$ and $j\in \bbZ$.
    There exists a bijection \[
        Y^+_j\colon 
        \left\{\begin{tabular}{@{} c @{}}
            $\bfP\in \SPD(w)$ \\
            with no red checkers \\
            in column $j+1$
        \end{tabular}\right\}
        \overset\sim\longrightarrow
        \left\{\begin{tabular}{@{} c @{}}
            $\bfQ\in \SPD(w)$ \\
            with no red checkers \\
            in column $j$
        \end{tabular}\right\}
    \]
    satisfying the following properties: if $Y^+_j\bfP = \bfQ$, then
    \begin{itemize}
        \item[(i)] $\bfP$ and $\bfQ$ differ only in columns $j$ and $j+1$,
        \item[(ii)] $\bfP$ and $\bfQ$ have the same number of black checkers in each row,
        \item[(iii)] $\bfQ$ has the same number of red checkers in column $j+1$ that $\bfP$ has in column $j$, and $\bfQ$ has no red checkers in column $j$.
    \end{itemize}
\end{proposition}

\begin{proof}
Suppose we are given a super pipe dream $\bfP\in \SPD(w)$ with no red checkers in column $j+1$.
Let $(i,j)$ denote the position of the lowest red checker in column $j$ (for the moment, assume one exists).
Form a $k$-ladder $L$ with $(i,j)$ as its SW corner and modify the arrangement of checkers in $L$ according to one of the following two rules:
\begin{itemize}
    \item[$\lozenge$]
    Suppose there is no checker at $(i,j+1)$.
    Let $(i',j)$ denote the first position above $(i,j)$ with no checkers.
    Take $L$ to be the big $k$-ladder with SW corner $(i,j)$ and NE corner $(i',j+1)$.
    
    Modify $L$ by first moving the red checkers at $(i,j)$ to $(i',j+1)$.
    Then modify each of the $k-2$ intermediate rows of $L$ by swapping the checker(s) appearing on the left with the checker(s) appearing on the right.
    Typical computations of this case are illustrated in Figure \ref{fig:column flow}.
    
    \item[$\blacklozenge$]
    Suppose there is a (black) checker at $(i,j+1)$.
    Take $L$ to be the $1$-ladder $\{(i,j),(i,j+1)\}$.

    Modify $L$ by swapping the checker(s) at $(i,j)$ with the checker at $(i,j+1)$.
\end{itemize}
This has the effect of moving all of the red checkers in column $j$ of $L$ to column $j+1$ without changing the quantity of black checkers in each row of $L$.
If there are still red checkers in column $j$, then repeat this process with the lowest red checker remaining in this column, and so forth, until we eventually obtain a super pipe dream $\bfQ$ which has no red checkers in column $j$.
Notice that each successive ladder is strictly higher than the previous, so the modifications done to $\bfP$ in each ladder to obtain $\bfQ$ can be performed in any order independent from one another.

Take $Y^+_j\bfP := \bfQ$.
It is clear that properties (i), (ii), and (iii) are satisfied by this construction.
To see why $\partial(\bfQ) = \partial(\bfP)$, observe that each modification done in a big ladder can be factored as a series of ladder moves; see Figure \ref{fig:column flow in ladders} to see this done for the second ladder in Figure \ref{fig:column flow}.
Since the modification to $\bfP$ done in each big ladder can be undone, $Y^+_j$ restricts to the desired bijection.
\end{proof}

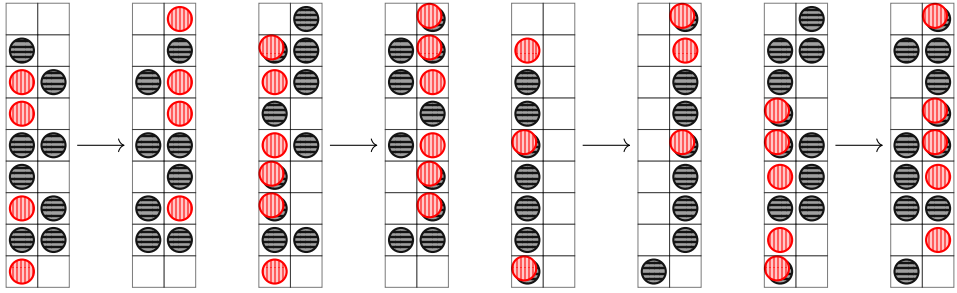
\begin{figure}
\begin{tikzpicture}[scale=0.42]
    % DIAGRAM A
    \begin{shift}
    \drawRCGraph{9}{2}{
        0,0,
        1,0,
        2,1,
        2,0,
        1,1,
        1,0,
        2,1,
        1,1,
        2,0
    }
    \draw[->] (2.75,-5) -- (4.25,-5);
    \begin{shift}[(4,0)]\drawRCGraph{9}{2}{
        0,2,
        0,1,
        1,2,
        0,2,
        1,1,
        0,1,
        1,2,
        1,1,
        0,0
    }\end{shift}
    \end{shift}
    
    % DIAGRAM B
    \begin{shift}[(8,0)]
    \drawRCGraph{9}{2}{
        0,1,
        3,1,
        2,1,
        1,0,
        2,1,
        3,0,
        3,0,
        1,1,
        2,0
    }
    \draw[->] (2.75,-5) -- (4.25,-5);
    \begin{shift}[(4,0)]\drawRCGraph{9}{2}{
        0,3,
        1,3,
        1,2,
        0,1,
        1,2,
        0,3,
        0,3,
        1,1,
        0,0
    }\end{shift}
    \end{shift}

    % DIAGRAM C
    \begin{shift}[(16,0)]
    \drawRCGraph{9}{2}{
        0,0,
        2,0,
        1,0,
        1,0,
        3,0,
        1,0,
        1,0,
        1,0,
        3,0
    }
    \draw[->] (2.75,-5) -- (4.25,-5);
    \begin{shift}[(4,0)]\drawRCGraph{9}{2}{
        0,3,
        0,2,
        0,1,
        0,1,
        0,3,
        0,1,
        0,1,
        0,1,
        1,0
    }\end{shift}
    \end{shift}

    % DIAGRAM D
    \begin{shift}[(24,0)]
    \drawRCGraph{9}{2}{
        0,1,
        1,1,
        1,0,
        3,0,
        3,1,
        2,1,
        1,1,
        2,0,
        3,0
    }
    \draw[->] (2.75,-5) -- (4.25,-5);
    \begin{shift}[(4,0)]\drawRCGraph{9}{2}{
        0,3,
        1,1,
        0,1,
        0,3,
        1,3,
        1,2,
        1,1,
        0,2,
        1,0
    }\end{shift}
    \end{shift}
\end{tikzpicture}

\caption{
    Typical modifications to $\bfP$ done in a big ladder $L$ when acted on by a flow operator $Y_j^+$.
}
\label{fig:column flow}
\end{figure}

Write $Y^-_{j+1}$ for the inverse of $Y^+_j$, which acts by flowing red checkers in column $j+1$ into column $j$.

For $j\in \bbZ$, define $Y^+_{\geq j} := Y^+_jY^+_{j+1}\cdots$, by which we mean $Y^+_{\geq j}\bfP := Y^+_jY^+_{j+1}\cdots Y^+_M\bfP$ where $M > j$ is large enough so that all of the red checkers in $\bfP$ are in or west of column $M$.
In words, $Y^+_{\geq j}$ acts by flowing all red checkers in or east of column $j$ to the right by one column.
This gives a bijection \[
    Y^+_{\geq j}\colon
    \SPD(w)
    \overset\sim\longrightarrow
    \left\{\begin{tabular}{@{} c @{}}
        $\bfQ\in \SPD(w)$ \text{with no} \\
        red checkers in column $j$
    \end{tabular}\right\}
\] with inverse given by $Y^-_{\geq j+1} := \cdots Y^-_{j+2}Y^-_{j+1}$, defined analogously.

Finally, set $Y^+ := \cdots Y^+_{-1} Y^+_0 Y^+_1\cdots$ i.e. $Y^+\bfP := Y^+_{\geq m}\bfP$ where $m$ is small enough so that all of the red checkers $\bfP$ are in columns on or east of $m$.
Thus $Y^+$ acts by flowing every red checker to the right by one column (see Figure \ref{fig:flow example}).
This gives a bijection $\SPD(w) \overset\sim\longrightarrow \SPD(w)$ with inverse $Y^- := \cdots Y^-_1 Y^-_0 Y^-_{-1} \cdots$.

If $m\leq M$ are as above for a super pipe dream $\bfP$, then we can think of $Y^+\bfP$ as being obtained as the final pipe dream in the chain \[
\begin{tikzcd}
    \bfP = Y^+_{\geq M+1}\bfP \arrow[r, "Y^+_M"]
    & Y^+_{\geq M}\bfP \arrow[r,"Y^+_{M-1}"]
    & \cdots \arrow[r,"Y^+_{m+1}"]
    & Y^+_{\geq m} \bfP = Y^+\bfP
\end{tikzcd}.
\]
We refer to the link $Y^+_{\geq j+1}\bfP \overset{Y^+_j}\longrightarrow Y^+_{\geq j}\bfP$ as \textit{step $j$} of $Y^+\bfP$.
This is the step where the red checkers in column $j$ flow into column $j+1$.

Extend each of the operators $Y^\pm_j$, $Y^\pm_{\geq j}$, $Y^\pm$ linearly to an endomorphism of $\bbZ\SPD$, the free abelian group on $\SPD$; if an operator is not defined on all of $\SPD$, then declare it to be $0$ on those $\bfP$ outside its domain.
There are corresponding endomorphisms of the polynomial ring $\bbZ[\beta][x;y]$ (which we also denote $Y^\pm_j$, $Y^\pm_{\geq j}$, $Y^\pm$) intertwining with the weight map $\wt$.
For example, the endomorphism $Y^+_j$ of $\bbZ[\beta][x;y]$ is given by the evaluations $y_j \mapsto y_{j+1}$ and $y_{j+1}\mapsto 0$ and satisfies $\wt(Y^+_j\bfP) = Y^+_j\wt(\bfP)$ for all $\bfP\in \SPD$.

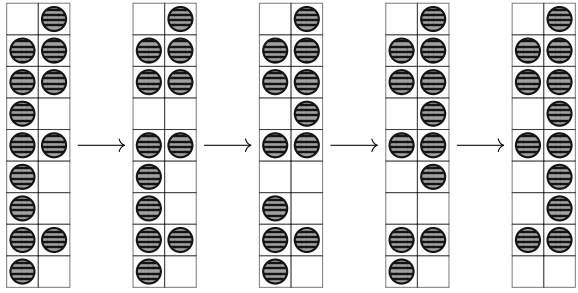
\begin{figure}
\begin{tikzpicture}[scale=0.42]
    % DIAGRAM 1
    \begin{shift}
    \drawRCGraph{9}{2}{
        0,1,
        1,1,
        1,1,
        1,0,
        1,1,
        1,0,
        1,0,
        1,1,
        1,0
    }
    \draw[->] (2.75,-5) -- (4.25,-5);
    \end{shift}
    
    % DIAGRAM 2
    \begin{shift}[(4,0)]
    \drawRCGraph{9}{2}{
        0,1,
        1,1,
        1,1,
        0,0,
        1,1,
        1,0,
        1,0,
        1,1,
        1,0
    }
    \draw[->] (2.75,-5) -- (4.25,-5);
    \end{shift}
    
    % DIAGRAM 3
    \begin{shift}[(8,0)]
    \drawRCGraph{9}{2}{
        0,1,
        1,1,
        1,1,
        0,1,
        1,1,
        0,0,
        1,0,
        1,1,
        1,0
    }
    \draw[->] (2.75,-5) -- (4.25,-5);
    \end{shift}
    
    % DIAGRAM 4
    \begin{shift}[(12,0)]
    \drawRCGraph{9}{2}{
        0,1,
        1,1,
        1,1,
        0,1,
        1,1,
        0,1,
        0,0,
        1,1,
        1,0
    }
    \draw[->] (2.75,-5) -- (4.25,-5);
    \end{shift}
    
    \begin{shift}[(16,0)]
    \drawRCGraph{9}{2}{
        0,1,
        1,1,
        1,1,
        0,1,
        1,1,
        0,1,
        0,1,
        1,1,
        0,0
    }
    \end{shift}
\end{tikzpicture}

\caption{
    The modification done to the underlying pipe dream of the second ladder in Figure \ref{fig:column flow} broken up as a series of ladder moves.
}
\label{fig:column flow in ladders}
\end{figure}

\subsection{Flowing black checkers and properties of flowing}
\label{subsect:flow properties}

By symmetry, there is also an operation $X^+_i$ for each $i\in \bbZ$ which takes a super pipe dream $\bfP$ with no $x$-crosses in row $i+1$ and ``flows'' the $x$-crosses in row $i$ into row $i+1$.
Indeed, mirroring the construction for $Y^+_j$ yields the same result as setting $X^+_i\bfP := (Y^+_i(\bfP^\dagger))^\dagger$.
We record this observation in the following proposition:

\begin{figure}
\begin{tikzpicture}[scale=0.45]
    % DIAGRAM 1
    \begin{shift}
    \draw[very thick] (1.5,-0.5) -- (1.5,-6.5);
    \draw[very thick] (0.5,-0.5) -- (6.5,-0.5);
    \drawRCGraph{6}{6}{
        9,1,0,0,2,0,
        0,3,2,1,2,3,
        0,2,0,3,0,0,
        0,1,3,0,0,1,
        0,2,2,0,1,0,
        0,0,0,0,0,0
    }
    \drawNorthLabels{0}{1}{0,1,2,3,4,5}
    \drawWestLabels{1}{0}{1,2,3,4,5,6}
    \node at (3.5,-7.5) {$\bfP = X^+_{\geq 6}\bfP$};
    \draw[->] (7,-4) -- (9,-4)
    node[midway,above] {$X^+_5$};
    \end{shift}

    % DIAGRAM 2
    \begin{shift}[(9,0)]
    \draw[very thick] (1.5,-0.5) -- (1.5,-6.5);
    \draw[very thick] (0.5,-0.5) -- (6.5,-0.5);
    \drawRCGraph{6}{6}{
        9,1,0,0,2,0,
        0,3,2,1,2,3,
        0,2,0,3,0,0,
        0,1,3,0,0,1,
        0,2,2,0,0,0,
        0,0,0,1,0,0
    }
    \node at (3.5,-7.5) {$X^+_{\geq 5}\bfP$};
    \draw[->] (7,-4) -- (9,-4)
    node[midway,above] {$X^+_4$};
    \end{shift}

    % DIAGRAM 3
    \begin{shift}[(18,0)]
    \draw[very thick] (1.5,-0.5) -- (1.5,-6.5);
    \draw[very thick] (0.5,-0.5) -- (6.5,-0.5);
    \drawRCGraph{6}{6}{
        9,1,0,0,2,0,
        0,3,2,1,2,3,
        0,2,0,3,0,0,
        0,2,2,0,0,0,
        0,1,3,0,1,0,
        0,0,0,1,0,0
    }
    \node at (3.5,-7.5) {$X^+_{\geq 4}\bfP$};
    \draw[->] (7,-4) -- (8,-4);
    \end{shift}

    % DIAGRAM 4
    \begin{shift}[(0,-9)]
    \draw[->] (-2,-4) -- (0,-4)
    node[midway,above] {$X^+_3$};
    \draw[very thick] (1.5,-0.5) -- (1.5,-6.5);
    \draw[very thick] (0.5,-0.5) -- (6.5,-0.5);
    \drawRCGraph{6}{6}{
        9,1,0,0,2,0,
        0,3,2,1,2,3,
        0,2,0,2,0,0,
        0,2,3,0,0,0,
        0,1,3,0,1,0,
        0,0,0,1,0,0
    }
    \node at (3.5,-7.5) {$X^+_{\geq 3}\bfP$};
    \draw[->] (7,-4) -- (9,-4)
    node[midway,above] {$X^+_2$};
    \end{shift}

    % DIAGRAM 5
    \begin{shift}[(9,-9)]
    \draw[very thick] (1.5,-0.5) -- (1.5,-6.5);
    \draw[very thick] (0.5,-0.5) -- (6.5,-0.5);
    \drawRCGraph{6}{6}{
        9,1,0,0,2,0,
        0,2,0,2,0,2,
        1,3,2,1,2,0,
        0,2,3,0,0,0,
        0,1,3,0,1,0,
        0,0,0,1,0,0
    }
    \node at (3.5,-7.5) {$X^+_{\geq 2}\bfP$};
    \draw[->] (7,-4) -- (9,-4)
    node[midway,above] {$X^+_1$};
    \end{shift}

    % DIAGRAM 6
    \begin{shift}[(18,-9)]
    \draw[very thick] (1.5,-0.5) -- (1.5,-6.5);
    \draw[very thick] (0.5,-0.5) -- (6.5,-0.5);
    \drawRCGraph{6}{6}{
        9,2,0,0,2,0,
        0,1,0,2,0,2,
        1,3,2,1,2,0,
        0,2,3,0,0,0,
        0,1,3,0,1,0,
        0,0,0,1,0,0
    }
    \node at (3.5,-7.5) {$X^+_{\geq 1}\bfP = X^+\bfP$};
    \end{shift}
\end{tikzpicture}

\caption{
    Computation of $X^+\bfP$ for the super pipe dream $\bfP$ from Figure \ref{fig:flow example}.
}
\label{fig:flow duality}
\end{figure}
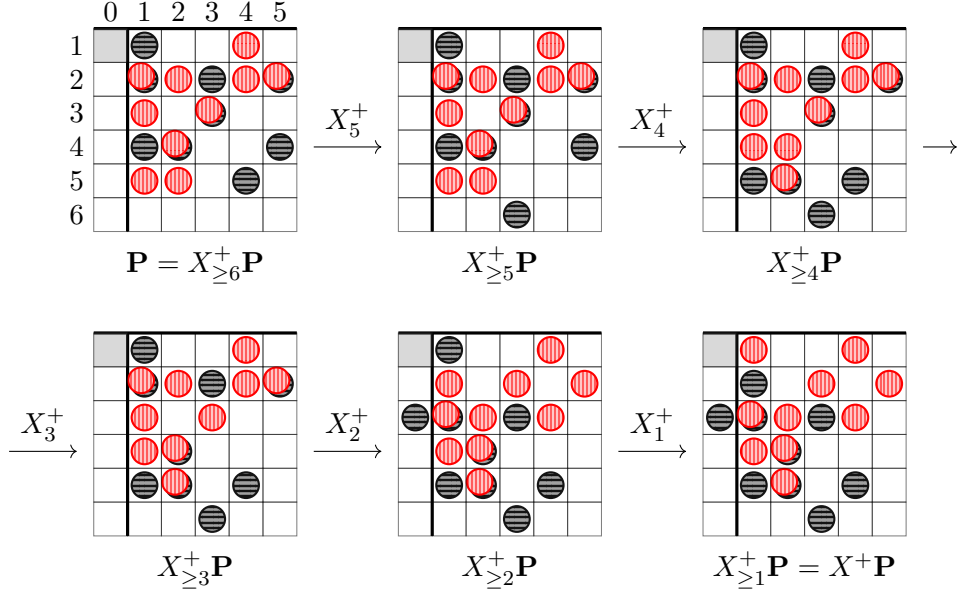

\begin{proposition}
Let $w\in S_\infty$ and $i\in \bbZ$.
There exists a bijection \[
    X^+_i\colon 
    \left\{\begin{tabular}{@{} c @{}}
        $\bfP\in \SPD(w)$ \\
        with no black checkers \\
        in row $i+1$
    \end{tabular}\right\}
    \overset\sim\longrightarrow
    \left\{\begin{tabular}{@{} c @{}}
        $\bfQ\in \SPD(w)$ \\
        with no black checkers \\
        in row $i$
    \end{tabular}\right\}
\]
satisfying the following properties: if $X^+_i\bfP = \bfQ$, then
\begin{itemize}
    \item[(i)] $\bfP$ and $\bfQ$ differ only in rows $i$ and $i+1$,
    \item[(ii)] $\bfP$ and $\bfQ$ have the same number of red checkers in each column,
    \item[(iii)] $\bfQ$ has the same number of black checkers in row $i+1$ that $\bfP$ has in row $i$, and $\bfQ$ has no black checkers in row $i$.
\end{itemize}
\end{proposition}

Similarly, we obtain operators $X^\pm_i$, $X^\pm_{\geq i}$, and $X^\pm$ acting on both the space $\bbZ\SPD$ and the polynomial ring $\bbZ[\beta][x;y]$.
We will speak of \textit{step $i$} of $X^+\bfP$ when referencing the link $X^+_{\geq i}\bfP \overset{X^+_i}\longrightarrow X^+_{\geq i+1}\bfP$ in the chain \[
\begin{tikzcd}
    \bfP = X^+_{\geq M+1}\bfP \arrow[r, "X^+_M"]
    & X^+_{\geq M}\bfP \arrow[r,"X^+_{M-1}"]
    & \cdots \arrow[r,"Y^+_{m+1}"]
    & X^+_{\geq m} \bfP = X^+\bfP
\end{tikzcd}.
\]
An example is given in Figure \ref{fig:flow duality}.

We next collect some immediate properties of the flow operators $X^+$ and $Y^+$. 
To help with this, define a \textit{shift} automorphism $\sigma$ first on $\PD$ by \[
    \sigma(P) := \{(i+1,j-1) \mid (i,j)\in P\}
\] and then on $\SPD$ by $\sigma(\bfP) := (\sigma(P_x),\sigma(P_y))$.
Notice that $\sigma(\bfP^\dagger) = (\sigma^{-1}(\bfP))^\dagger$.
We also let $\sigma$ act on $\bbZ[\beta][x;y]$ by $x_i\mapsto x_{i+1}$ and $y_j\mapsto y_{j-1}$ so that $\wt(\sigma(\bfP)) = \sigma(\wt(\bfP))$.

\begin{proposition}
\label{prop:flow properties}
Let $\bfP\in \SPD$.
\begin{itemize}
\item[(i)]
If $P_y = \varnothing$, then $X^+\bfP = \sigma(\bfP)$ and $Y^+\bfP = \bfP$. \\
If $P_x = \varnothing$, then $Y^+\bfP = \sigma^{-1}(\bfP)$ and $X^+\bfP = \bfP$.

\item[(ii)]
$X^+(\sigma(\bfP)) = \sigma(X^+\bfP)$
\quad and \quad
$Y^+(\sigma(\bfP)) = \sigma(Y^+\bfP)$

\item[(iii)]
$X^+(\bfP^\dagger) = (Y^+\bfP)^\dagger$
\quad and \quad
$Y^+(\bfP^\dagger) = (X^+\bfP)^\dagger$.
\end{itemize}
\end{proposition}

\subsection{Flowing for reduced pipe dreams}
\label{subsect:reduced flowing}

There is a simpler description of the operator $Y^+_j$ for reduced pipe dreams.
First, define an auxiliary operator $Y'_j$ acting reduced $\bfP\in \SPD_0$. 
Let $(i,j)$ denote the highest red checker in column $j$ of $\bfP$ (if none exist, set $Y'_j\bfP := \bfP$).
Form a $k$-ladder $L$ with SW corner $(i,j)$ and modify $\bfP$ in $L$ as follows:
\begin{itemize}
    \item[$\lozenge$]
    If $(i,j+1)$ has no checkers, let $(i',j)$ be the first tile above $(i,j)$ with no checkers and take $L$ to be the big $k$-ladder with SW and NE corners $(i,j)$ and $(i',j+1)$, respectively.
    
    Modify $L$ by first moving the red checker from $(i,j)$ to $(i',j+1)$.
    Then for each of the $k-2$ intermediate rows of $L$, if this row only has one checker (a black checker appearing on the left), move this checker over to the right.
    
    \item[$\blacklozenge$]
    If $(i,j+1)$ has a checker, take $L$ to be the $1$-ladder $\{(i,j),(i,j+1)\}$.
    
    Modify $\bfP$ in $L$ by swapping the red checker at $(i,j)$ with the checker at $(i,j+1)$.
\end{itemize}

Now, given $\bfP\in \SPD_0$ with no red checkers in column $j+1$, we can obtain $Y^+_j\bfP$ as by iteratively applying $Y'_j$ to $\bfP$.
More precisely, $Y^+_j\bfP = (Y'_j)^k \bfP$ where $k$ is chosen to be at least the number of red checkers in column $j$ of $\bfP$.
This gives a factorization of $Y^+_j$ into a series of steps which only moves one red checker at a time.

To see why this works, suppose $Y^+_j \bfP$ is obtained from $\bfP$ through modifications in ladders $L_1,\dots,L_r$ with $L_1$ being the highest ladder and $L_r$ being the lowest.
Let $k_i$ denote the number of red checkers in $L_i$.
Then by induction on $0\leq s\leq r$, $(Y'_i)^{k_1+\dots+k_s}\bfP$ is the pipe dream obtained by modifying ladders $L_1,\dots,L_s$ as in the definition of $Y^+_j$.
Figure \ref{fig:column flow in ticks} illustrates an example of the case when $L_s$ is a big ladder, with the case where $L_s$ is a $1$-ladder being clear.

\begin{figure}
\begin{tikzpicture}[scale=0.42]
    % DIAGRAM 1
    \begin{shift}
    \drawRCGraph{9}{2}{
        0,0,
        1,0,
        2,1,
        2,0,
        1,1,
        1,0,
        2,1,
        1,1,
        2,0
    }
    \draw[->] (2.75,-5) -- (4.25,-5)
    node[midway,above] {$Y_j'$};
    \end{shift}

    % DIAGRAM 2
    \begin{shift}[(4,0)]
    \drawRCGraph{9}{2}{
        0,0,
        1,0,
        1,2,
        2,0,
        1,1,
        1,0,
        2,1,
        1,1,
        2,0
    }
    \draw[->] (2.75,-5) -- (4.25,-5)
    node[midway,above] {$Y_j'$};
    \end{shift}

    % DIAGRAM 3
    \begin{shift}[(8,0)]
    \drawRCGraph{9}{2}{
        0,2,
        0,1,
        1,2,
        0,0,
        1,1,
        1,0,
        2,1,
        1,1,
        2,0
    }
    \draw[->] (2.75,-5) -- (4.25,-5)
    node[midway,above] {$Y_j'$};
    \end{shift}

    % DIAGRAM 4
    \begin{shift}[(12,0)]
    \drawRCGraph{9}{2}{
        0,2,
        0,1,
        1,2,
        0,0,
        1,1,
        1,0,
        1,2,
        1,1,
        2,0
    }
    \draw[->] (2.75,-5) -- (4.25,-5)
    node[midway,above] {$Y_j'$};
    \end{shift}
    
    % DIAGRAM 5
    \begin{shift}[(16,0)]\drawRCGraph{9}{2}{
        0,2,
        0,1,
        1,2,
        0,2,
        1,1,
        0,1,
        1,2,
        1,1,
        0,0
    }
    \end{shift}
    
\end{tikzpicture}

\caption{
    A factorization of the modification done to the first ladder in Figure \ref{fig:column flow} as successive applications of $Y'_j$.
}
\label{fig:column flow in ticks}
\end{figure}
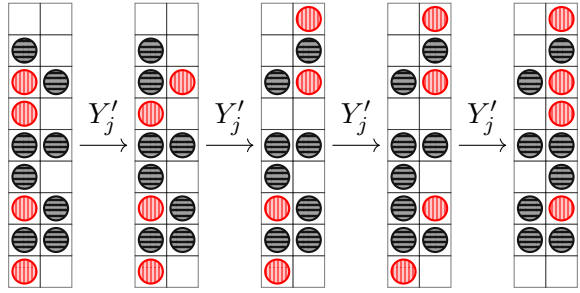

\section{Symmetry theorem for flowing}
\label{sect:flow symmetry}

Notice that $X^+$ and $\sigma Y^+$ both act on the polynomial ring $\bbZ[\beta][x;y]$ the evaluations $x_i\mapsto x_{i+1}$.
It follows that for $\bfP\in \SPD$, the super pipe dreams $X^+\bfP$ and $\sigma(Y^+\bfP)$ have the same monomial weight, which we can witness by comparing Figures \ref{fig:flow example} and \ref{fig:flow duality}.
Indeed, a much stronger statement can be said for our example: \textit{equality holds on the level of pipe dreams.}
Surprisingly, this turns out to be the case in general.

\begin{proposition}
\label{prop:flow symmetry}
Let $\bfP\in \SPD$.
Then $\sigma(Y^+(\bfP)) = X^+(\bfP)$.
\end{proposition}

\noindent
\textit{Remark:}\quad
It is not true in general that if $\bfP$ and $\bfQ$ are super pipe dreams with $\wt(\bfP) = \wt(\bfQ)$, then $\bfP = \bfQ$.
See Figure \ref{fig:same weight, different pipes} for a counter-example.
\vspace{4pt}

\begin{figure}
\begin{tikzpicture}[scale=0.45]
    \begin{shift}
    \draw[very thick] (0.5,-2.5) -- (0.5,-0.5) -- (2.5,-0.5);
    \drawRCGraph{2}{2}{
        2,1,
        1,2
    }
    \drawNorthLabels{0}{1}{1,2}
    \drawWestLabels {1}{0}{1,2}
    \end{shift}

    \begin{shift}[(6,0)]
    \draw[very thick] (0.5,-2.5) -- (0.5,-0.5) -- (2.5,-0.5);
    \drawRCGraph{2}{2}{
        1,2,
        2,1
    }
    \drawNorthLabels{0}{1}{1,2}
    \drawWestLabels {1}{0}{1,2}
    \end{shift}
\end{tikzpicture}
\caption{
    Two super pipe dreams with weight $x_1x_2y_1y_2$.
    The underlying pipe dream of each is the same.
}
\label{fig:same weight, different pipes}
\end{figure}
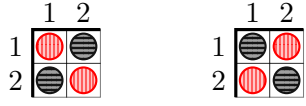

To prove Proposition \ref{prop:flow symmetry}, we will need the help of two technical lemmas.

\begin{lemma}
\label{lem:SW corner of big ladder}
    Let $\bfP\in \SPD(w)$.
    Suppose $(i,j)$ is the SW corner of a big ladder $L$ during step $j$ of $Y^+\bfP$.
    Let $(i_0,j)$ (resp. $(i,j_0)$) denote the first tile north (resp. east) of $(i,j)$ with no checkers in $\bfP$ (so $(i_0,j+1)$ is the NE corner of $L$).
    
    \begin{itemize}
    \item[(i)]
    For each $j<j'< j_0$, $(i,j')$ has no checkers in $Y^+_{\geq j'} \bfP$ and, hence, appears in a big ladder $L'$ during step $j'$ of $Y^+\bfP$.
    
    \item[(ii)]
    If $(i,j)$ appears in a big chute $C$ during step $i$ of $X^+\bfP$, then $(i,j)$ is the NE corner of $C$.
    
    \item[(iii)]
    For each $i_0 < i'\leq i$, $(i',j)$ has a red checker in $X^+_{\geq i'} \bfP$.
    \end{itemize}
\end{lemma}

\begin{proof}
We first show (i) by induction on $j< j'< j_0$.
If $j' = j+1$, then $(i,j+1)$ has no checkers in $Y^+_{\geq j+1} \bfP$ since it is the SE corner of the big ladder $L$ during step $j$ of $Y^+\bfP$.
If $j' > j+1$, then by induction $(i,j'-1)$ has no checkers in $Y^+_{\geq j'-1}$ and appears in a big ladder $L'$ during this step; these conditions imply that $(i,j')$ must not have any checkers in $Y^+_{\geq j'}\bfP$.
In any case, $(i,j')$ has checkers in $\bfP$ but not in $Y^+\bfP$, so it must appear in big ladder during step $j+1$ of $Y^+\bfP$.
This completes (i).

We will next show that (iii) is a consequence of (ii), so assume that (ii) holds.
We proceed by downwards induction on $i\geq i'> i_0$.
If $i' = i$, notice that $(i,j)$ being the SW corner of a big ladder $L$ during step $j$ of $Y^+\bfP$ means that $(i,j)$ has a red checker in $\bfP$.
Hence the only way for $(i,j)$ to not contain a red checker in $X^+_{\geq i}$ is if it is contained in a middle column of a big chute during step $i$ of $X^+\bfP$, but this cannot happen by (ii).

Now, let $i_0<i'<i$ and assume there is a red checker at $(i'+1,j)$ in $X^+_{\geq i'+1}\bfP$.
There is also is a checker at $(i',j)$ in $\bfP$ and, hence, in $X^+_{\geq i'+1}\bfP$ since $L$ is a big ladder.
If $(i',j)$ is contained in a chute during step $i'$ of $X^+\bfP$, then the checkers at $(i',j)$ and $(i'+1,j)$ in $X^+_{\geq i'+1}\bfP$ get swapped in $X^+_{\geq i'}\bfP$, so $(i',j)$ has a red checker in $X^+_{\geq i'}\bfP$.
Otherwise $(i',j)$ is not contained in a chute during step $i'$ of $X^+\bfP$, in which case $(i',j)$ must already have a red checker in $\bfP$ and will continue to have this red checker in $X^+_{\geq i'}\bfP$.
This completes the proof that (ii) implies (iii).

It remains to show (ii).
Before doing so, pick any linear extension of the partial order on $\bbZ^2$ where $(i,j)\geq (i',j')$ if and only if $(i,j)$ is weakly northwest of $(i',j')$.
We will show (ii) by induction on the positions $(i,j)$ satisfying the hypotheses of the lemma using the above order.

Suppose $(i,j)$ is contained in a big chute $C$ during step $i$ of $X^+\bfP$.
Let $(i,j_1)$ denote the NE corner of $C$, so $j\leq j_1< j_0$ and $(i,j_1)$ has a black checker in $\bfP$.
Our goal is to show that $j_1=j$.

If instead $j_1> j$, then by (i) $(i,j_1)$ is contained in a big ladder $L'$ during step $j_1$ of $Y^+\bfP$ (note that no such $L'$ can exist for the base case of the induction since its SE corner would be southeast of $(i,j)$, so in this case $j_1 = j$).
Write $(i_1,j_1)\in P_y$ denote the SW corner of $L'$.
We cannot have $i_1 = i$ since otherwise $(i,j_1)$ would still have a black checker in $Y^+_{\geq j_1}\bfP$ in contradiction to (i), so $i_1 > i$.

By the induction hypothesis, $(i_1,j_1)$ satisfies both (ii) and, hence, (iii).
In particular, $(i+1,j_1)$ has a red checker in $X^+_{\geq i+1}\bfP$.
On the other hand, we know $(i+1,j_1)$ has no checker in $X^+_{\geq i+1}\bfP$, since this is the SE corner of the big chute $C$ during step $i$ of $X^+\bfP$.
This gives us the desired contradiction, so we must have $j_1=j$.
This completes the induction for (ii).
\end{proof}

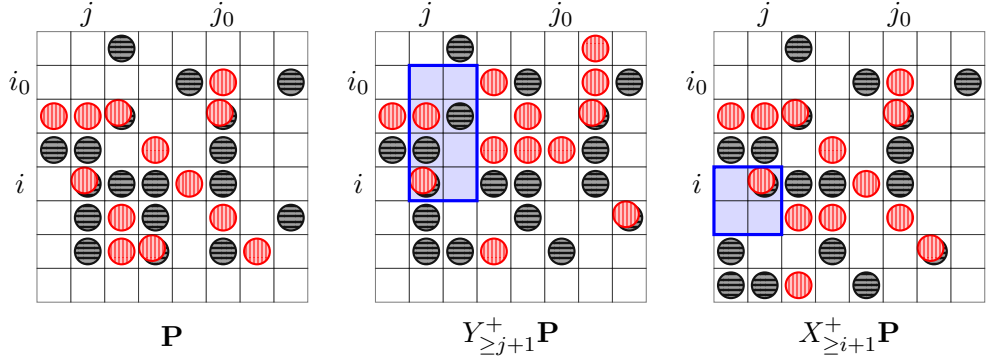
\begin{figure}
\begin{tikzpicture}[scale=0.45]
    \begin{shift}
    \drawRCGraph{8}{8}{
        0,0,1,0,0,0,0,0,
        0,0,0,0,1,2,0,1,
        2,2,3,0,0,3,0,0,
        1,1,0,2,0,1,0,0,
        0,3,1,1,2,1,0,0,
        0,1,2,1,0,2,0,1,
        0,1,2,3,0,1,2,0,
        0,0,0,0,0,0,0,0
    }
    \drawWestLabels{1}{0}{,$i_0$,,,$i$,}
    \drawNorthLabels{0}{1}{,$j$,,,,$j_0$,}
    \node at (4.5,-9.5) {$\bfP$};
    \end{shift}

    \begin{shift}[(10,0)]
    \fill[blue,opacity=0.15] (1.5,-1.5) rectangle ++(2,-4);
    \drawRCGraph{8}{8}{
        0,0,1,0,0,0,2,0,
        0,0,0,2,1,0,2,1,
        2,2,1,0,2,0,3,0,
        1,1,0,2,2,2,1,0,
        0,3,0,1,1,0,1,0,
        0,1,0,0,1,0,0,3,
        0,1,1,2,0,1,0,0,
        0,0,0,0,0,0,0,0
    }
    \drawWestLabels{1}{0}{,$i_0$,,,$i$,}
    \drawNorthLabels{0}{1}{,$j$,,,,$j_0$,,}
    \node at (4.5,-9.5) {$Y^+_{\geq j+1}\bfP$};
    \draw[blue,very thick] (1.5,-1.5) rectangle ++(2,-4);
    \end{shift}

    \begin{shift}[(20,0)]
    \fill[blue,opacity=0.15] (0.5,-4.5) rectangle ++(2,-2);
    \drawRCGraph{8}{8}{
        0,0,1,0,0,0,0,0,
        0,0,0,0,1,2,0,1,
        2,2,3,0,0,3,0,0,
        1,1,0,2,0,1,0,0,
        0,3,1,1,2,1,0,0,
        0,0,2,2,0,2,0,0,
        1,0,0,1,0,0,3,0,
        1,1,2,0,1,0,0,0
    }
    \drawWestLabels{1}{0}{,$i_0$,,,$i$,}
    \drawNorthLabels{0}{1}{,$j$,,,,$j_0$,,}
    \node at (4.5,-9.5) {$X^+_{\geq i+1}\bfP$};
    \draw[blue,very thick] (0.5,-4.5) rectangle ++(2,-2);
    \end{shift}
\end{tikzpicture}

\caption{
    Illustration of Lemma \ref{lem:SW corner of big ladder}.
    The shaded region in the second (resp. third) diagram is the ladder (resp. chute) containing $(i,j)$ during step $j$ of $Y^+\bfP$ (resp. step $i$ of $X^+\bfP$).
}
\label{fig:SW corner of big ladder}
\end{figure}

\begin{lemma}
\label{lem:not in big ladder}
Let $\bfP \in \SPD(w)$.
Suppose there is a checker at some position $(i,j)$ in $\bfP$ which is not contained in a big ladder during step $j$ of $Y^+\bfP$.
Then there is no checker at $(i+1,j)$ in $X^+_{\geq i+1}\bfP$, and $(i,j)$ is contained in a big chute $C$ during step $i$ of $X^+\bfP$.
If there is a red checker at $(i,j)$ in $\bfP$, then $(i,j)$ is not the NE corner of $C$.
\end{lemma}

\begin{proof}
We again induct on those $(i,j)$ satisfying the hypotheses of the lemma using any linear extension of the northeast partial ordering on $\bbZ^2$.

We first show that there is no checker at $(i+1,j)$ in $X^+_{\geq i+1}\bfP$.
This is clear if $(i+1,j)$ is unoccupied in $\bfP$, so we may assume that there is a checker at $(i+1,j)$ in $\bfP$.
$(i+1,j)$ cannot be part of a big ladder during step $j$ of $Y^+\bfP$, since otherwise $(i,j)$ would have to be contained in this same big ladder.
% (in particular, this gives the base of the induction).

Applying the induction hypothesis to $(i+1,j)$, there is no checker at $(i+2,j)$ in $X^+_{\geq i+2}\bfP$ and $(i+1,j)$ is contained in a big chute during step $i+1$ of $X^+\bfP$.
If there is only a black checker at $(i+1,j)$ in $\bfP$ and, hence, in $X^+_{\geq i+2}\bfP$, then $(i+1,j)$ will become vacant in $X^+_{\geq i+1}\bfP$.
Otherwise there is a red checker at $(i+1,j)$ in both $\bfP$ and $X^+_{\geq i+2}\bfP$, in which case $(i+1,j)$ appears in one of the middle columns of $C$, so again $(i+1,j)$ becomes vacant in $X^+_{\geq i+1}\bfP$.

\begin{figure}
\begin{tikzpicture}[scale=0.45]
    \begin{shift}
    \drawRCGraph{8}{8}{
        1,0,0,0,1,0,0,0,
        3,2,0,3,1,3,1,0,
        0,3,1,1,0,1,2,0,
        0,1,2,1,2,3,0,1,
        0,0,3,3,0,1,0,0,
        0,2,1,0,1,2,3,1,
        0,1,1,0,3,0,1,0,
        0,2,0,0,0,0,2,0
    }
    \drawNorthLabels{0}{1}{,,$j$,,,$j_1$,,}
    \drawWestLabels{1}{0}{,,,$i$,,,,,}
    \node at (4.5,-9.5) {$\bfP$};
    \end{shift}

    \begin{shift}[(10,0)]
    \fill[blue,opacity=0.15] (2.5,-3.5) rectangle ++(2,-1);
    \drawRCGraph{8}{8}{
        1,0,0,0,3,0,2,2,
        3,2,0,1,3,0,3,1,
        0,3,1,0,1,0,1,0,
        0,1,2,1,1,2,0,1,
        0,0,3,1,0,3,0,2,
        0,2,1,0,1,1,2,3,
        0,1,1,0,1,0,0,1,
        0,2,0,0,0,0,0,0
    }
    \drawNorthLabels{0}{1}{,,$j$,,,$j_1$,,}
    \drawWestLabels{1}{0}{,,,$i$,,,,,}
    \node at (4.5,-9.5) {$Y^+_{\geq j+1}\bfP$};
    \draw[blue,very thick] (2.5,-3.5) rectangle ++(2,-1);
    \end{shift}

    \begin{shift}[(20,0)]
    \fill[blue,opacity=0.15] (0.5,-3.5) rectangle ++(6,-2);
    \drawRCGraph{8}{8}{
        1,0,0,0,1,0,0,0,
        3,2,0,3,1,3,1,0,
        0,3,1,1,0,1,2,0,
        0,1,2,1,2,3,0,1,
        0,0,0,1,0,0,0,0,
        0,3,3,0,3,0,2,0,
        1,2,0,1,1,2,3,0,
        1,1,0,1,0,0,1,0
    }
    \drawNorthLabels{0}{1}{,,$j$,,,$j_1$,,}
    \drawWestLabels{1}{0}{,,,$i$,,,,,}
    \node at (4.5,-9.5) {$X^+_{\geq i+1}\bfP$};
    \draw[blue,very thick] (0.5,-3.5) rectangle ++(6,-2);
    \end{shift}
\end{tikzpicture}

\caption{
    Illustration of Lemma \ref{lem:not in big ladder}. The shaded regions are as in Figure \ref{fig:SW corner of big ladder}.
}
\label{fig:not in big ladder}
\end{figure}
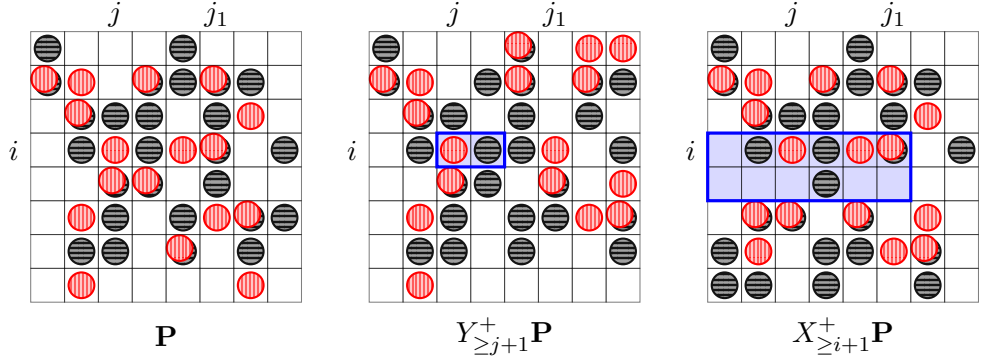

We next show that $(i,j)$ is part of a big chute during step $i$ of $X^+\bfP$.
If there is only a black checker at $(i,j)$ in $\bfP$, then $(i,j)$ is contained in some chute $C$ during step $i$ of $X^+\bfP$.
Since $(i+1,j)$ is a bump in $X^+_{\geq i+1} \bfP$, $C$ must be a big chute.

It remains to consider the case where there is a red checker at $(i,j)$ in $\bfP$.
Let $L$ be the ladder containing $(i,j)$ during step $j$ of $Y^+\bfP$.
By hypothesis, $L$ must be a $1$-ladder, so $(i,j+1)$ only has a black checker in $Y^+_{\geq j+1}\bfP$.

Let $j_1>j$ be minimal so that $(i,j_1)$ either is not contained in a ladder during step $j_1$ of $Y^+\bfP$ or appears as the SW corner of a big ladder during step $j_1$ of $Y^+\bfP$.
Then for all $j<j'<j_1$, the checkers at positions $(i,j')$ and $(i,j'+1)$ are swapped during step $j'$ of $Y^+\bfP$ (these positions either form a $1$-ladder or the middle row of a big ladder).
This means that the checkers at $(i,j_1)$ in $Y^+_{\geq j_1}\bfP$ are the same as the checkers at $(i,j+1)$ in $Y^+_{\geq j+1}\bfP$, and we saw the latter consisted of only a black checker.

We claim that $(i,j_1)$ also has a black checker in $\bfP$.
This is clear if $(i,j_1)$ is not contained in a ladder during step $j_1$ of $Y^+\bfP$, so assume that it is contained in such a ladder $L'$.
By construction, $L'$ is a big ladder with SW corner $(i,j_1)$.
Since $(i,j_1)$ has a black checker in $Y^+_{\geq j_1}\bfP$, it must have had one already in $\bfP$, so we get the claim.

In particular, $(i,j_1)$ is contained in a chute $C$ during step $i$ of $X^+\bfP$.
If $C$ is a big chute, which must be the case if $(i+1,j_1)$ is unoccupied in $X^+_{\geq i+1}\bfP$, then $(i,j)$ would also be contained in $C$ since there are checkers at each $(i,j')$ in $\bfP$ where $j\leq j'\leq j_1$.
Note that in this case, $(i,j)$ is not the NE corner of $C$.
Hence to finish the proof, it suffices to show that $(i+1,j_1)$ is unoccupied in $X^+_{\geq i+1} \bfP$.
We examine three cases (with some overlap between the first two):

\begin{itemize}
\item
If $(i+1,j_1)$ is unoccupied in $\bfP$, then it will continue to be in $X^+_{\geq i+1} \bfP$.

\item
If $(i,j_1)$ is not contained in a ladder during step $j_1$ of $Y^+\bfP$, then by the induction hypothesis $(i+1,j_1)$ will be unoccupied in $X^+_{\geq i+1} \bfP$.

\item
Suppose that there is a checker at $(i+1,j_1)$ in $\bfP$, and that $(i,j_1)$ is contained in a ladder $L'$ during step $j_1$ of $Y^+\bfP$ (which, by construction, is a big ladder with SW corner $(i,j_1)$).
Then $(i+1,j_1)$ cannot also be in a big ladder during step $j_1$ of $Y^+\bfP$, since such a ladder would overlap with $L'$.
The second paragraph of this proof verbatim with $j$ replaced by $j'$ shows that $(i+1,j_1)$ has no checkers in $X^+_{\geq i+1}\bfP$.
\qedhere
\end{itemize}
\end{proof}

\begin{proof}[Proof of Proposition \ref{prop:flow symmetry}]
We will show that if there is a red checker at $(i,j+1)$ in $Y^+\bfP$, then there is also one at $(i+1,j)$ in $X^+\bfP$.
Since $Y^+\bfP$ and $X^+\bfP$ have the same number of red checkers, it would follow that $(X^+\bfP)_y = \sigma((Y^+\bfP)_y)$.
The same would hold for the black checkers by the computation
\begin{align*}
    (X^+(\bfP^\dagger))_x
    &= ((Y^+\bfP)^\dagger)_x
    \tag{Proposition \ref{prop:flow properties}(iii)} \\
    &= ((Y^+\bfP)_y)^t
    \\
    &= \big(\sigma^{-1}((X^+\bfP)_y)\big)^t
    \tag{since $(X^+\bfP)_y = \sigma((Y^+\bfP)_y)$} \\
    &= \big((\sigma^{-1}(X^+\bfP))^\dagger\big)_x
    \\
    &= \sigma\big(((X^+\bfP)^\dagger)_x\big)
    \tag{since $\sigma(\bfP^\dagger) = (\sigma^{-1}(\bfP))^\dagger$}\\
    &= \sigma\big((Y^+(\bfP^\dagger))_x\big).
    \tag{Proposition \ref{prop:flow properties}(iii)}
\end{align*}
Now replace $\bfP^\dagger$ with $\bfP$.

Let $(i,j+1)$ be the location of a red checker in $Y^+\bfP$.
This checker was placed here during step $j$ of $Y^+\bfP$, at which point it is not moved again.
Let $L$ denote the ladder containing $(i,j+1)$ during this step.

First consider the case where $L$ is a big ladder.
Let $i_0$ and $i_1$ denote the top and bottom rows of $L$ (so $i_0\leq i< i_1$).
Lemma \ref{lem:SW corner of big ladder}(iii) (using $(i,j) = (i_1,j)$ and $i'=i+1$) tells us that there is a red checker at $(i+1,j)$ in $X^+_{\geq i+1}\bfP$.
We may assume that $(i+1,j)$ is contained in a chute $C$ during step $i$ of $X^+\bfP$, since otherwise $(i+1,j)$ will continue to have a red checker in $X^+\bfP$.
We examine two cases:
\begin{itemize}
\item 
Suppose $i=i_0$.
Then $(i,j) = (i_0,j)$ is empty in $\bfP$, since this is the NW corner of the big ladder $L$.
It follows that $C$ is a big chute with SW corner $(i+1,j)$, so there will still be a red checker at $(i+1,j)$ in $X^+_{\geq i}\bfP$ and, hence, in $X^+\bfP$.

\item
Suppose $i> i_0$.
The checkers at $(i,j)$ and $(i,j+1)$ are swapped during step $j$ of $Y^+\bfP$ (since these positions form a middle row of $L$), so $(i,j)$ has a red checker in $\bfP$.
Then $(i,j)$ and $(i+1,j)$ both have red checkers in $X^+_{\geq i+1}\bfP$, so $C$ must be a big chute with these positions occurring as a middle column.
After step $i$ of $X^+\bfP$, the checkers at $(i,j)$ and $(i+1,j)$ are thus swapped, so $(i+1,j)$ has a red checker in $X^+_{\geq i}\bfP$ and in $X^+\bfP$.
\end{itemize}

Now consider the case where $L$ is a $1$-ladder.
Then $(i,j)$ is occupied by a red checker in $\bfP$ and, hence, in $X^+_{\geq i+1}\bfP$.
Lemma \ref{lem:not in big ladder} guarantees that $(i,j)$ appears in a middle column of a big chute $C$ during step $i$ of $X^+\bfP$.
The checkers at $(i,j)$ and $(i+1,j)$ in $X^+_{\geq i+1}\bfP$ are thus swapped when passing to $X^+_{\geq i}\bfP$, so $(i+1,j)$ is a $y$-cross in $X^+_{\geq i}\bfP$ and therefore in $X^+\bfP$.
This completes the proof.
\end{proof}

\section{Pipe dream rectification and the Cauchy identities}
\label{sect:cauchy rectification}

In this section, we provide an algorithm that takes a super pipe dream $\bfW\in \SPD(w)$ and separates it into two single pipe dreams $V\in \PD(v)$ and $U\in \PD(u)$ for a decomposition $w = u^{-1}*v$.
We call this process pipe dream \textit{rectification}.
We will see that rectification preserves weights in such a way to give a bijective proof of the Cauchy identities for Schubert and Grothendieck polynomials (Theorems \ref{thm:cauchy identity} and \ref{thm:K-cauchy identity}).

\subsection{Pipe dream rectification}
\label{subsect:rectification}

We begin by giving the definition of pipe dream rectification.

\begin{proposition}
\label{prop:rectification}
Let $w\in S_\infty$.
There is a bijection \[
    \Rect\colon \SPD(w)
    \overset\sim\longrightarrow
    \coprod_{
        w = u^{-1} * v
    } \PD(v) \times \PD(u)
\] which preserves weights in the sense that if $(V,U) = \Rect(\bfW)$, then $\wt(\bfW) = \beta^{\ell(u) + \ell(v) - \ell(w)}\wt(V)\wt(U^\dagger)$.
We refer to $\Rect(\bfW)$ as the \textit{rectification} of $\bfW$.
\end{proposition}
As a reminder, we identify $U$ with the super pipe dream $(U,\varnothing)$, so in the above proposition $U^\dagger = (U,\varnothing)^\dagger = (\varnothing,U^t)$.

\begin{proof}
Assume that $w\in S_n$ and that the black (resp. red) checkers in $\bfW$ are in rows at or below $i_0$ (resp. columns at or right of $j_0$).
For each $m\geq 0$, the red checkers in $(Y^+)^m\bfW$ get flowed into the region $\{(i,j) \mid 1\leq i+j-1< n,i\geq i_0+m\}$ where as the black checkers remain in the region $\{(i,j) \mid 1\leq i+j-1<n, j\leq j_0\}$.

Picking $m\geq 0$ large enough so that these regions are disjoint, we reach a super pipe dream $\bfW' := (Y^+)^m\bfW$ where the red checkers are positioned northeast of the black checkers.
In this case, $\mathbf{a}(\bfW') = \mathbf{a}(W'_y)\mathbf{a}(W'_x)$, so we get a decomposition $w = u^{-1} * v$ where $u^{-1} = \partial(W'_y)$ and $v = \partial(W'_x)$.

Take $\Rect(\bfW) := (U,V)$ where $V = \bfW'_x\in \PD(v)$ and $U^t = \sigma^m(\bfW'_y) \in \PD(u^{-1})$ (see Figure \ref{fig:rectification}).
This is well-defined regardless of which $m$ we picked, since applying $Y^+$ any further to $\bfW'$ is same as applying $\sigma^{-1}$ to the red checkers of $\bfW'$.
By construction, $\bfW$ and $V$ have the same number of black checkers in each row, and $U$ and $W_y$ have the same number of red checkers in each column, so $\Rect$ preserves weights in the prescribed way.
It is clear that the above process can be inverted, so $\Rect$ is a bijection.
\end{proof}

\begin{figure}
\begin{tikzpicture}[scale=0.45]
    % LEFT COLUMN
    \begin{shift}
    \draw[very thick] (0.5,-5.5) -- (0.5,-0.5) -- (5.5,-0.5);
    \drawRCGraph{5}{5}{
        1,0,0,2,0,
        3,2,1,2,3,
        2,0,3,0,0,
        1,3,0,0,1,
        2,2,0,1,0
    }
    \drawNorthLabels{0}{1}{1,2,3,4,5}
    \drawWestLabels{1}{0}{1,2,3,4,5}
    \node at (3,-6.25) {$\bfW$};

    % xvec = [(1,4),(2,1),(2,3),(2,5),(3,3),(4,1),(4,2),(4,5),(5,4)]
    % yvec = [(1,1),(2,1),(2,2),(2,4),(2,5),(3,1),(3,3),(4,2),(5,1),(5,2)]
    
    \draw[->] (6,-3) -- (9,-3)
    node[midway,above] {$(Y^+)^5$};
    \end{shift}

    % MIDDLE COLUMN
    \begin{shift}[(10,2.5)]
    \draw[very thick] (0.5,-0.5) -- (0.5,-10.5);
    \draw[very thick] (0.5,-5.5) -- (10.5,-5.5);
    \drawRCGraph{10}{10}{
        9,9,9,9,9,2,0,0,2,0,
        9,9,9,9,0,2,0,2,0,2,
        9,9,9,0,0,2,2,0,2,0,
        9,9,0,0,0,0,2,0,0,0,
        9,0,0,0,0,0,2,0,0,0,
        0,0,0,1,0,0,0,0,0,0,
        1,1,1,0,0,0,0,0,0,0,
        0,0,0,1,0,0,0,0,0,0,
        0,1,1,0,1,0,0,0,0,0,
        0,0,0,1,0,0,0,0,0,0
    }
    \drawNorthLabels{0}{1}{1,2,3,4,5,6,7,8,9,10}
    \drawWestLabels{1}{0}{-4,-3,-2,-1,0,1,2,3,4,5}
    % \draw[very thick,dashed,color=x color 1]
    % (0.55,-10.45) rectangle ++(4.9,4.9);
    % \draw[very thick,dashed,color=y color 1]
    % (5.55,-5.45) rectangle ++(4.9,4.9);
    \end{shift}

    % RIGHT COLUMN
    \begin{shift}[(23,-3.5)]
    \draw[very thick] (0.5,-5.5) -- (0.5,-0.5) -- (5.5,-0.5);
    \drawRCGraph{5}{5}{
        0,0,0,1,0,
        1,1,1,0,0,
        0,0,0,1,0,
        0,1,1,0,1,
        0,0,0,1,0
    }
    \drawNorthLabels{0}{1}{1,2,3,4,5}
    \drawWestLabels{1}{0}{1,2,3,4,5}
    \node at (6.5,-3) {$V$};
    \end{shift}
    
    \begin{shift}[(23,3.5)]
    \draw[very thick] (0.5,-5.5) -- (0.5,-0.5) -- (5.5,-0.5);
    \drawRCGraph{5}{5}{
        2,0,0,2,0,
        2,0,2,0,2,
        2,2,0,2,0,
        0,2,0,0,0,
        0,2,0,0,0
    }
    \drawNorthLabels{0}{1}{1,2,3,4,5}
    \drawWestLabels{1}{0}{1,2,3,4,5}
    \node at (6.5,-3) {\color{y color 1}$U^\dagger$};
    \end{shift}

    % ARROW
    \draw[->] (3,-7.5) -- (3,-10) -- (26,-10) -- (26,-9.5);
    \node at (14.5,-10.5) {$\Rect$};
\end{tikzpicture}

\caption{
    Rectification of the super pipe dream $\bfW$ (then called $\bfP$) from Figure \ref{fig:flow example}.
    Here $\Rect(\bfW) = (V,U) \in \PD^+(v)\times \PD^+(u)$ for $u = 2351764$ and $v = 152374698$.
    One can verify that $w = 273561498 = u^{-1}*v$.
}
\label{fig:rectification}
\end{figure}
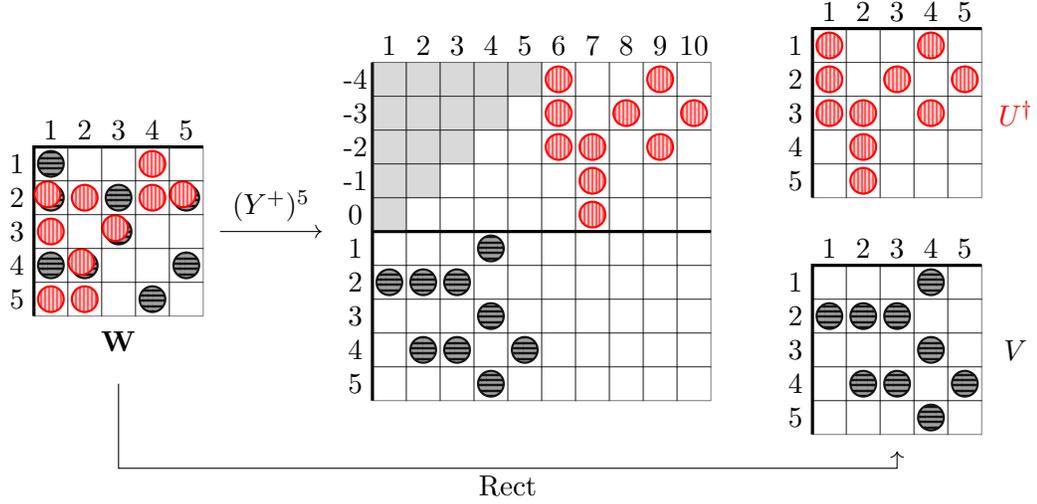

Notice that we could have dually found $m\geq 0$ such that the red checkers in $\bfW' := (Y^-)^m\bfW$ are all southwest of the black checkers, then obtaining pipe dreams $V = W'_x$ and $U = \sigma^{-m}(W'_y)$.
This turns out to be a distinct process which we call pipe dream \textit{corectification}:

\begin{proposition}
\label{prop:corectification}
Let $w\in S_\infty$.
There is a bijection \[
    \coRect\colon \SPD(w)
    \overset\sim\longrightarrow
    \coprod_{
        w = u^{-1} * v
    } \PD(v) \times \PD(u)
\] which preserves weights as in Proposition \ref{prop:rectification}.
We call the pair $\coRect(\bfW) := (V,U)$ the \textit{corectification} of $\bfW$.
\end{proposition}

An immediate consequence of Proposition \ref{prop:flow symmetry} is the following symmetry result enjoyed by both rectification and corectification.

\begin{proposition}
\label{prop:symmetry theorem}
Let $w\in S_\infty$ and $\bfW\in \SPD(w)$.
\begin{itemize}
    \item[(i)]
    If $\Rect(\bfW) = (V,U)$, then $\Rect(\bfW^\dagger) = (U,V)$.
    \item[(ii)]
    If $\coRect(\bfW) = (V,U)$, then $\coRect(\bfW^\dagger) = (U,V)$.
\end{itemize}
\end{proposition}

\begin{proof}
We prove (i), with (ii) being similar.
Let $m\geq 0$ be large enough for both $\bfW$ and $\bfW^\dagger$ as in the definition of $\Rect$.
Write $\bfW' = (Y^+)^m\bfW$ so that $V = W'_x$ and $U^t = \sigma^m(W'_y)$.
Then by Propositions \ref{prop:flow properties} and \ref{prop:flow symmetry}, we have \begin{align*}
    \bfW''
    := (Y^+)^m(\bfW^\dagger)
    = ((X^+)^m\bfW)^\dagger
    = (\sigma^m(\bfW'))^\dagger,
\end{align*}
so $
    W''_x = (\sigma^m(W'_y))^t = U
$ and $
    W''_y = (\sigma^m(W'_x))^t = \sigma^{-m}((W'_x)^t) = \sigma^{-m}(V^t)
$.
It follows that $\Rect(\bfW^\dagger) = (W''_x,(\sigma^m(W''_y))^t) = (U,V)$.
\end{proof}

On a first pass, there may not seem to be a reason to prefer rectification over corectification, but it turns out that the former has nicer properties.\footnote{
    Perhaps \textit{in-corectification} would have been a better name for this map.
}
Proposition \ref{prop:rectification preserves ordinary} is the key step to seeing this.
We first prove the following lemma:

\begin{lemma}
\label{lem:Y+ preserves ordinary}
Let $w\in S_\infty$, $\bfP\in \SPD(w)$, and $k\geq 0$.
Write $\bfQ := Y^+\bfP$.
The following conditions are equivalent:
\begin{itemize}
    \item
    $P_x$ and $\sigma^k(P_y)$ are both ordinary.
    \item 
    $Q_x$ and $\sigma^{k+1}(Q_y)$ are both ordinary.
\end{itemize}
\end{lemma}

\begin{proof}
Note first that in general, if $\bfP = Y^+\bfQ$, then
$P_x\sseq \{(i,j) \mid i\geq 1\} \iff Q_x\sseq \{(i,j) \mid i\geq 1\}$ and $P_y\sseq \{(i,j)\mid j\geq 1+k\} \iff Q_y\sseq \{(i,j)\mid j\geq 1+(k+1)\}$.
Furthermore, $P_x\cup P_y\sseq \{(i,j) \mid i\geq 1-k\} \implies Q_x\cup Q_y\sseq \{(i,j)\mid i\geq 1-(k+1)\}$ since for each $j\in\bbZ$, the top checker(s) in columns $j$ of $Y^+_{\geq j+1}\bfP$ will be in at or below row $1-k$, so each ladder during step $j$ of $Y^+\bfP$ has top row $\leq 1-(k+1)$.

In particular, either condition in the proposition implies that $\bfQ = Y^+_{k+1}Y^+_{k+2}\dots Y^+_M\bfP$ for $M$ sufficiently large, so $\bfQ$ differs from $\bfP$ only in columns $k$ and greater.
This shows that $Q_x$ and $\sigma^{k+1}(Q_y)$ are ordinary assuming the first condition, and $P_x$ is ordinary assuming the second condition.

It remains to show that $\sigma^k(P_y)$ is ordinary assuming the second condition.
If it is not, the issue would be a red checker in $\bfP$ at some position $(i,j)$ with $i \leq 1-(k+1)$.
There are no black checkers in row $i$ in $\bfP$ and, hence, in $Y^+_{\geq j+1}\bfP$, so $(i,j)$ must be contained in a big ladder $L$ during step $j$ of $Y^+\bfP$.
$L$ necessarily has top row $i_0 < i \leq 1-(k+1)$.
But then $(i_0,j)$ would be a checker in both $Y^+_{\geq j}\bfP$ and, hence, $\bfQ$.
This contradicts the hypothesis.
\end{proof}

\begin{proposition}
\label{prop:rectification preserves ordinary}
Let $w\in S_\infty$ and $\bfW\in \SPD(w)$.
Write $(V,U) = \Rect(\bfW)$.
Then $\bfW$ is ordinary if and only if both $V$ and $U$ are ordinary.
\end{proposition}

\noindent
\textit{Remark:}\quad
The conclusion of Proposition \ref{prop:rectification preserves ordinary} is not necessarily true if we replace $\Rect$ with $\coRect$.
See Figure \ref{fig:corectification} for an example.
\vspace{4pt}

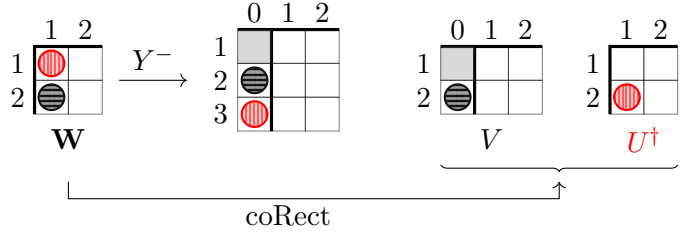
\begin{figure}
\begin{tikzpicture}[scale=0.45]
    % DIAGRAM 1
    \begin{shift}
    \draw[very thick] (0.5,-2.5) -- (0.5,-0.5) -- (2.5,-0.5);
    \drawRCGraph{2}{2}{
        2,0,
        1,0
    }
    \drawNorthLabels{0}{1}{1,2}
    \drawWestLabels{1}{0}{1,2}
    \node at (1.5,-3.25) {$\bfW$};
    
    \draw[->] (3,-1.5) -- (5,-1.5) node[midway,above] {$Y^-$};
    \end{shift}

    % DIAGRAM 2
    \begin{shift}[(6,0.5)]
    \draw[very thick] (1.5,-3.5) -- (1.5,-0.5);
    \draw[very thick] (0.5,-0.5) -- (3.5,-0.5);
    \drawRCGraph{3}{3}{
        9,0,0,
        1,0,0,
        2,0,0
    }
    \drawNorthLabels{0}{1}{0,1,2}
    \drawWestLabels{1}{0}{1,2,3}
    \end{shift}

    % DIAGRAM 3
    \begin{shift}[(12,0)]
    \draw[very thick] (1.5,-2.5) -- (1.5,-0.5);
    \draw[very thick] (0.5,-0.5) -- (3.5,-0.5);
    \drawRCGraph{2}{3}{
        9,0,0,
        1,0,0
    }
    \drawNorthLabels{0}{1}{0,1,2}
    \drawWestLabels{1}{0}{1,2}
    \node at (2,-3.25) {$V$};
    \end{shift}
    
    % DIAGRAM 4
    \begin{shift}[(17,0)]
    \draw[very thick] (0.5,-2.5) -- (0.5,-0.5) -- (2.5,-0.5);
    \drawRCGraph{2}{2}{
        0,0,
        2,0
    }
    \drawNorthLabels{0}{1}{1,2}
    \drawWestLabels{1}{0}{1,2}
    \node at (1.5,-3.25) {\color{y color 1}$U^\dagger$};
    \end{shift}

    % ARROW
    \draw[->] (1.5,-4.5) -- (1.5,-5) -- (16,-5) -- (16,-4.5);
    \node at (8,-5.5) {$\coRect$};
    \draw [decorate,decoration={brace,mirror}] (12.5,-4) -- (19.5,-4);
\end{tikzpicture}

\caption{
    An ordinary super pipe dream $\bfW\in \SPD^+$ whose corectification $\coRect(\bfW) = (V,U)$ does not lie in $\PD^+\times \PD^+$.
}
\label{fig:corectification}
\end{figure}

\begin{proof}
By induction with Lemma \ref{lem:Y+ preserves ordinary}, we see that $\bfW$ is ordinary if and only if $((Y^+)^k\bfW)_x$ and $(\sigma^k((Y^+)^k\bfW))_y$ are ordinary for all $k\geq 0$.
Taking $k=m$ as in the definition of $\Rect$ gives the result.
\end{proof}

\subsection{A bijective proof of the Cauchy identity for Grothendieck polynomials}
\label{subsect:cauchy identities}

Proposition \ref{prop:rectification preserves ordinary} allows us to obtain a bijective proof the Cauchy identity for Grothendieck polynomials (Theorem \ref{thm:K-cauchy identity}).

\begin{proof}[Proof of Theorem \ref{thm:K-cauchy identity}]
By Proposition \ref{prop:rectification preserves ordinary}, $\Rect$ restricts to a bijection
\[
    \Rect\colon \SPD^+(w)
    \overset\sim\longrightarrow
    \coprod_{
        w = u^{-1} * v
    } \PD^+(v) \times \PD^+(u).
\]
Furthermore, Proposition \ref{prop:rectification} guarantees that if $\bfW\in \SPD(w)$ with $(V,U) = \Rect(\bfW)$, then $\wt(\bfW) = \beta^{\ell(u)+\ell(v)-\ell(w)}\wt(V)\wt(U^\dagger)$.
The proposition now follows by the observation that \[
    \frakG^{(\beta)}_v(x)\frakG^{(\beta)}_u(y)
    = \sum_{\substack{
        V\in \PD^+(v) \\
        U\in \PD^+(u)
    }} \wt(V)\wt(U^\dagger) \qedhere
\]
\end{proof}

We next show the Cauchy identity for Schubert polynomials (Theorem \ref{thm:cauchy identity}).
The key step is to check that $\Rect$ preserves reducedness.

\begin{proposition}
\label{prop:rectification preserves reducedness}
Let $w\in S_\infty$ and $\bfW\in \SPD(w)$.
Write $(V,U)\in \PD(v)\times \PD(u)$ for either $\Rect(\bfW)$ or $\coRect(\bfW)$.
Then $\bfW$ is reduced if and only if $\ell(w) = \ell(u) + \ell(v)$ and both $V$ and $U$ are reduced.
\end{proposition}

\begin{proof}
If $\bfW$ is reduced, then $\ell(w) = |W_x| + |W_y| = |V| + |U| \geq \ell(v) + \ell(u) \geq \ell(w)$, so in fact $|V| = \ell(v)$, $|U| = \ell(u)$, and $\ell(w) = \ell(v) + \ell(u)$.

Conversely, if $U$ and $V$ are reduced and $\ell(w) = \ell(u) + \ell(v)$, then $\ell(w) = \ell(u) + \ell(v) = |V| + |U| = |W_x| + |W_y| \geq |W_x\cup W_y| \geq \ell(w)$, so $|W_x| + |W_y| = \ell(w)$.
\end{proof}

\begin{proof}[Proof of Theorem \ref{thm:cauchy identity}]
Combining Propositions \ref{prop:rectification preserves ordinary} and \ref{prop:rectification preserves reducedness},
we get a bijection \[
    \Rect\colon \SPD^+_0(w)
    \overset\sim\longrightarrow
    \coprod_{
        w\doteq u^{-1}v
    } \PD^+_0(v) \times \PD^+_0(u)
\] which is weight preserving by Proposition \ref{prop:rectification}.
\end{proof}

\section{Further applications}
\label{sect:applications}

\subsection{Derivatives of Grothendieck polynomials and Macdonald's reduced word formula}
\label{subsect:nabla identities}

Consider the differential operator \[
    \nabla^{(\beta)} := 
    \sum_{i=1}^\infty (1+\beta x_i) \frac\partial{\partial x_i},
\] acting on the polynomial ring $\bbZ[\beta][x]$.
The following positive formula for $\nabla^{(\beta)}$ acting on a Grothendieck polynomial was observed by Pechenik, Speyer, and Weigandt in \cite{HamakerPechenikSpeyerWeigandt20}.
\begin{proposition}
\label{prop:K-nabla identity}
Let $w\in S_\infty$. Then \[
    \nabla^{(\beta)} \frakG_w^{(\beta)}(x)
    = \sum_{
        w = s_k*v
    }
    \beta^{\delta_{w,v}}k\cdot \frakG_v^{(\beta)}(x)
\] where $\delta_{w,v} =
\left\{\begin{smallmatrix}
    1 & & w=v \vspace{2pt}\\
    0 & & w\neq v
\end{smallmatrix}\right.$.
\end{proposition}

\begin{proof}
This identity can be viewed as a shadow of the $y$-degree 1 component of Theorem \ref{thm:K-cauchy identity} applied to $w$:
\begin{equation}
\label{eq:K-cauchy y-deg 1}
    \sum_{\substack{
        \bfP\in \SPD^+(w)\\ \abs{P_y} = 1
    }} \wt(\bfP)
    = \sum_{
        w = s_k*v
    } \beta^{\delta_{w,v}}
    \frakG_v^{(\beta)}(x) \frakS_{s_k}(y).
\end{equation}
By specializing the $y$-variables to all be $1$ (and using that $\frakS_{s_k}(x) = x_1+\cdots+x_k$), the RHS of Equation \ref{eq:K-cauchy y-deg 1} becomes the RHS of the identity.

The LHS of Equation \ref{eq:K-cauchy y-deg 1} under this specialization becomes \begin{equation*}
    \sum_{\substack{
        \bfP\in \SPD^+(w) \\
        \abs{P_y} = 1
    }} \beta^{\abs{P_x}+1-\ell(w)} x^{P_x}
    = \sum_{\substack{
        Q\in \PD^+(w) \\
        (i,j)\in Q
    }}
    \left(\frac1{x_i} + \beta\right)\wt(Q)
\end{equation*}
which follows by noting that for each pair $(Q,(i,j))$ counted on the right, there are two super pipe dreams $\bfP = (P_x,P_y)$ counted on the left satisfying $P_x\cup P_y = Q$ and $P_y = \{(i,j)\}$ (namely, $P_x = Q\sm \{(i,j)\}$ and $P_x = Q$).
It is simple to check that the corresponding weights $\beta^{\abs{P_x}+1-\ell(w)}x^{P_x}$ of these super pipe dreams are $\frac1{x_i}\wt(Q)$ and $\beta\wt(Q)$, respectively.
Since \[
    \nabla^{(\beta)}\wt(Q) = \sum_{(i,j)\in P} \left(\frac1{x_i}+\beta\right)\wt(Q),
\] we see that the LHS of Equation \ref{eq:K-cauchy y-deg 1} specializes to the LHS of the identity.
\end{proof}

By viewing Proposition \ref{prop:K-nabla identity} as a graded piece of Theorem \ref{thm:K-cauchy identity},
our bijective proof of the latter specializes to a bijective proof of the former by restricting $\Rect$ to those $\bfP\in \SPD^+(w)$ with $\abs{P_y} = 1$.

By taking $\beta = 0$ in Proposition \ref{prop:K-nabla identity}, we recover a formula of Hamaker, Pechenik, Speyer, and Weigandt \cite{HamakerPechenikSpeyerWeigandt20} for the differential operator $
    \nabla
    := \nabla^{(0)}
    = \sum_{i\geq 1} \frac\partial{\partial x_i}
$ acting on a Schubert polynomial.
\begin{corollary}
\label{cor:nabla identity}
Let $w\in S_\infty$.
Then \[
    \nabla \frakS_w(x)
    = \sum_{
        w \doteq s_kv
    } k\cdot \frakS_v(x).
\]
\end{corollary}

\noindent
Alternatively, we can view this identity as coming from the $y$-degree $1$ component of Theorem \ref{thm:cauchy identity}.\footnote{
    This observation was communicated to the author by Anna Weigandt.
} A bijective proof is obtained by restricting $\Rect$ to those $\bfP\in \SPD^+_0(w)$ with $\abs{P_y} = 1$.

A consequence of Corollary \ref{cor:nabla identity} observed in \cite{HamakerPechenikSpeyerWeigandt20} is that by applying $\nabla^{\ell(w)}$ to the Schubert polynomial $\frakS_w(x)$, we recover the following reduced word identity of Macdonald \cite{Macdonald91}:

\begin{proposition}[Macdonald's reduced word formula]
\label{prop:macdonald formula}
Let $w\in S_\infty$ and write $\calR(w)$ for the set of reduced words for $w$.
Then \[
    \sum_{\mathbf{a} \in \calR(w)} a_1\cdots a_{\ell(w)}
    = \ell(w)! \cdot \frakS_w(1,\dots,1).
\]
\end{proposition}

\begin{proof}
A bijective proof can be obtained by iterating our bijective proof of Proposition \ref{cor:nabla identity}.
Explicitly, the RHS counts the number of ways to pick a pipe dream $P\in \PD^+(w)$ and then form a chain \[
    P = P_0\to P_1\to \cdots\to P_{\ell(w)} = \varnothing
\] where for $i\in [\ell(w)]$, the link $P_{i-1}\to P_i$ is formed by choosing one of the $\ell(w)-i+1$ checkers in $P_{i-1}$ to color red and then applying $\Rect$ to get a pair $(P_i,Q_i)$ where $\partial(P_{i-1}) \doteq s_{a_i}\partial(P_i)$ and $Q_i\in \PD^+_0(s_{a_i})$.
The LHS counts the equivalent data $\bfa = (a_1,\dots,a_{\ell(w)}) \in \calR(w)$ and $(Q_1,\dots,Q_{\ell(w)}) \in \PD^+_0(s_{a_1})\times\cdots\times \PD^+_0(s_{a_{\ell(w)}})$.
\end{proof}

\subsection{A restricted descent Pieri rule and an insertion algorithm on pipe dreams}
\label{subsect:pieri rule}

Fix $m\geq 1$.
For $w\in S_\infty$, let $w\ominus_m 1$ denote the permutation given by \[
   w\ominus_m 1(i) := \begin{cases}
        w(i)+1   & 1 \leq i \leq m \\
        1        & i = m+1 \\
        w(i-1)+1 & i \geq m+2
   \end{cases}
\]
Observe that $\ell(w\ominus_m 1) = \ell(w) + m$.
In fact, if $\bfa := (a_1<\cdots<a_m)$ denotes the first $m$ values of $w$, then one can verify that $w\ominus_m 1 \doteq \partial(\bfa)w$.

Given $P\in \PD^+(w)$ we can obtain a pipe dream $P\ominus_m 1\in \PD^+(w\ominus_m 1)$ by moving all of the checkers in $P$ to the right by $1$ and then placing $m$ checkers at the top of the first column.
If $\Des(w)\sseq [m]$, then this construction is bijective.

Consider the following $\beta$-deformation of the $k$\textsuperscript{th} elementary symmetric polynomial in $m$ variables: \[
    e_k^{(\beta)}(x_1,\dots,x_m)
    := \sum_{\substack{J\sseq I\sseq[m] \\ \abs{J} = k}}
    \beta^{\abs{I}-k} x_I
    = \sum_{r=k}^m \binom{r}{k} \beta^{r-k} e_r(x_1,\dots,x_m).
\]
Here we write $x_I = \prod_{i\in I} x_i$ whenever $I\sseq [m]$.
Note that the usual elementary symmetric polynomial $e_k(x_1,\dots,x_m)$ is recovered by taking $\beta = 0$.

The next proposition provides a ``Pieri rule'' for multiplying $e_k^{(\beta)}(x_1,\dots,x_m)$ by the Grothendieck polynomial $\frakG^{(\beta)}_w(x)$ of a permutation with descents no bigger than $m$.

\begin{proposition}
\label{prop:K-pieri rule}
Let $w\in S_\infty$ be with $\Des(w)\sseq [m]$.
Then \[
    e_k^{(\beta)}(x_1,\dots,x_m) \frakG_w^{(\beta)}(x)
    = \sum_{\substack{
        v, a_1 < \cdots < a_{m-k} \\
        w\ominus_m 1 = \partial(\bfa) * v
    }}
    \beta^{\ell(v)-\ell(w)-k} \frakG_v^{(\beta)}(x).
\]
\end{proposition}

\begin{proof}
We will derive this identity from Theorem \ref{thm:K-cauchy identity} applied to $w\ominus_m 1$
\begin{equation}
\label{eq:K-cauchy w-1}
    \frakG_{w\ominus_m 1}^{(\beta)}(x;y) = \sum_{
        w\ominus_m 1 = u^{-1}*v
    } \beta^{\ell(u)+\ell(v)-(\ell(w)+m)} \frakG_v^{(\beta)}(x) \frakG_u^{(\beta)}(y)
\end{equation}
by examining the $\bbZ[\beta][x]$-coefficients of the $y_1^{m-k}$ terms on each side.

We first do this for the RHS of Equation \ref{eq:K-cauchy w-1}.
If $(V,U)$ is a pair contributing to the $y_1^{m-k}$ term, then $U$ consists of $m-k$ checkers in the first row.
Hence the corresponding decomposition of $w\ominus_m 1$ looks like $\partial(\bfa)*v$ for some $\bfa = (a_1<\cdots<a_{m-k})$.
Since $U$ is necessarily reduced, the contribution of this term is $\beta^{\ell(v)-\ell(w)-k}\wt(V)y_1^{m-k}$, hence the coefficient of $y_1^{m-k}$ on the RHS of Equation \ref{eq:K-cauchy w-1} is precisely the RHS of the identity.

\begin{figure}
\begin{tikzpicture}[scale=0.45]
    \begin{shift}
    \node at (-3,-2) {$I = \{2,4,5\}$};
    \node at (-3,-4) {$J = \{2,5\}$};
    \draw[very thick] (0.5,-5.5) -- (0.5,-0.5) -- (5.5,-0.5);
    \drawRCGraph{5}{5}{
        1,1,0,1,1,
        0,1,1,0,0,
        0,0,0,0,0,
        1,0,1,1,0,
        1,0,0,0,0
    }
    \node at (3,-6.5) {$P$};
    \draw[->] (6.5,-3) -- (9.5,-3);
    \end{shift}
    
    \begin{shift}[(10,0)]
    \draw[very thick] (0.5,-5.5) -- (0.5,-0.5) -- (6.5,-0.5);
    \drawRCGraph{5}{6}{
        2,1,1,0,1,1,
        1,0,1,1,0,0,
        2,0,0,0,0,0,
        3,1,0,1,1,0,
        1,1,0,0,0,0
    }
    \node at (3.5,-6.5) {$\bfW$};
    \end{shift}
\end{tikzpicture}
\caption{
    Illustration of Proposition \ref{prop:K-pieri rule} with $w = 31648257$, $m = 5$ ($\Des(w) = \{1,3,5\}\sseq [5]$).
    The data $P\in \PD^+(w)$ and $J\sseq I\sseq [5]$ corresponds to the super pipe dream $\bfW\in \SPD^+(w\ominus_5 1)$ where $w\ominus_5 1 = 427591368$.
}
\label{fig:pieri construction}
\end{figure}
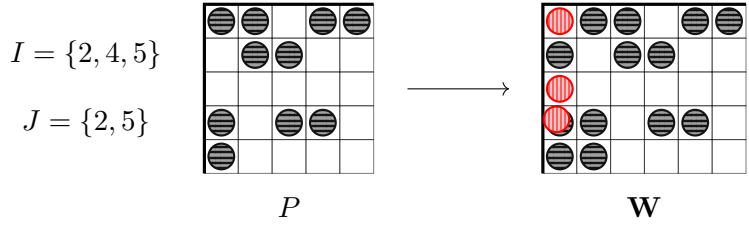

We now unpack the $y_1^{m-k}$ term on the LHS of Equation \ref{eq:K-cauchy w-1}.
The data of a super pipe dream $\bfW\in \SPD^+(w\ominus_m 1)$ contributing to this term (so $\bfW$ only has $m-k$ red checkers, all placed in the first column) is the same as the data of the pipe dream $P\in \PD^+(w)$ determining the underlying pipe dream $P\ominus_m 1$ of $\bfW$, as well as the sets $J\sseq I\sseq [m]$ where $I$ and $[m]\sm J$ encode, respectively, the rows containing black and red checkers in the first column of $\bfW$ (so $\abs{J} = k$); see Figure \ref{fig:pieri construction}.
With this notation, we have $\wt(\bfW) = \beta^{\abs{I}-k}x_Iy_1^{m-k}\wt(P)$.
This shows that the coefficient of $y_1^{m-k}$ on the LHS of Equation \ref{eq:K-cauchy w-1} is the product $e_k^{(\beta)}(x_1,\dots,x_m)\frakG_w^{(\beta)}(x)$.
\end{proof}

As with the differential identities in the previous section, the derivation of Proposition \ref{prop:K-pieri rule} from Theorem \ref{thm:K-cauchy identity} affords a bijective proof of the former through rectification.
The totality of these bijections obtained by varying $w$ and $k$ can be interpreted as constituting an ``$m$-insertion algorithm'' on the collection of pipe dreams for permutations $w$ with $\Des(w)\sseq [m]$.
We make this precise:

\begin{definition}[$m$-insertion]
\label{def:insertion}
Let $w\in S_\infty$ be with $\Des(w)\sseq [m]$.
Given $P\in \PD^+(w)$ and $J\sseq I\sseq [m]$, we can obtain a new pipe dream denoted $(I,J)\msert P$ as the output $V$ in the pair $\Rect(\bfW) = (V,U)$.
Here $\bfW$ is the super pipe dream corresponding to the triple $(P,I,J)$ as in the proof of Proposition \ref{prop:K-pieri rule} using $k = \abs{J}$.
\end{definition}

Note that by Proposition \ref{prop:K-pieri rule}, $(I,J)\msert P\in \PD^+(v)$ for some $w\ominus_m 1 = \partial(\bfa)*v$.
In particular, $v$ also satisfies $\Des(v)\sseq \Des(w \ominus_m 1) = \Des(w)\cup \{m\} \sseq [m]$.

We will use the shorthand notations $I\msert P$ for $(I,I)\msert P$ and $i\msert P$ for $\{i\}\msert P$.
Notice that $\wt((I,J)\msert P) = \beta^{|I|-|J|} x_I \wt(P)$.
In particular, $I\msert P$ is reduced whenever $P$ is reduced.

Setting $\beta=0$ in Proposition \ref{prop:K-pieri rule} recovers a Pieri rule for Schubert polynomials.
\begin{corollary}
\label{cor:pieri rule}
Let $w\in S_\infty$ be with $\Des(w)\sseq [m]$.
Then \[
    e_k(x_1,\dots,x_m) \frakS_w(x)
    = \sum_{\substack{
        v, a_1 < \cdots < a_{m-k} \\
        w\ominus_m 1 \doteq \partial(\bfa) v
    }}
    \frakS_v(x).
\]
\end{corollary}

\noindent
Alternatively, this identity follows by comparing the $y_1^{m-k}$ terms in Theorem \ref{thm:cauchy identity} applied to $w\ominus_m 1$.
The resulting bijective proof is obtained by restricting $\Rect$ to those $\bfW\in \SPD_0^+(w\ominus_m 1)$ with $m-k$ red checkers all in the first column.
The collection of these bijections corresponds to $m$-insertions of the form $I\msert P$ where $P$ is reduced.

We will study this insertion algorithm further in Section \ref{sect:RSK}.

\subsection{An identity of Stanley symmetric functions}
\label{subsect:Stanley symmetric}

Although corectification does not furnish a proof of the Cauchy identities (cf. Proposition \ref{prop:rectification preserves ordinary}), there are situations where it can have utility.
We will see one such instance in this section.

Say that a pipe dream $P\in \PD$ is \textit{stable} if $P\sseq \bbN\times \bbZ$.
Write $\PD^{\downarrow}$ for the collection of stable pipe dreams.

\begin{definition}[\cite{Stanley84}]
\label{def:stanley symmetric}
Let $w\in S_\infty$.
The \textit{Stanley symmetric function} and \textit{$K$-Stanley symmetric function}\footnote{
    These are also called \textit{stable Schubert/Grothendieck polynomials}, hence the terminology for stable pipe dreams.
} of $w$ are the formal power series \begin{align*}
    F_w(x) &= \sum_{P\in \PD^\downarrow_0(w)} \wt(P), &
    G_w^{(\beta)}(x) &= \sum_{P\in \PD^\downarrow(w)} \wt(P).
\end{align*}
\end{definition}

Rectification and corectification can both be used to obtain the following identity of $K$-Stanley symmetric functions.

\begin{figure}
\begin{tikzpicture}[scale=0.42]
    % DIAGRAM 1
    \begin{shift}
    \draw[very thick] (0.5,-0.5) -- (6.5,-0.5);
    \draw[very thick] (2.5,-0.5) -- (2.5,-4.5);
    \drawRCGraph{4}{6}{
        9,9,1,1,0,0,
        9,0,1,0,1,0,
        0,0,0,0,0,0,
        1,1,0,0,0,1
    }
    \node at (3.5,-5.5) {$V$};
    
    \begin{shift}[(7,0)]
    \draw[very thick] (0.5,-0.5) -- (1.5,-0.5);
    \draw[very thick] (0.5,-0.5) -- (0.5,-4.5);
    \drawRCGraph{4}{1}{
        0,
        0,
        0,
        2
    }
    \node at (1,-5.5) {\color{y color 1}$P_4^\dagger$};
    \end{shift}
    \end{shift}
    
    % DIAGRAM 2
    \draw[<-] (9,-2.5) -- (12,-2.5)
    node[midway,above] {$\coRect$};
    
    \begin{shift}[(12,0)]
    \draw[very thick] (0.5,-0.5) -- (6.5,-0.5);
    \draw[very thick] (2.5,-0.5) -- (2.5,-4.5);
    \drawRCGraph{4}{6}{
        9,9,1,1,0,0,
        9,0,1,0,1,0,
        0,0,2,0,0,0,
        1,0,1,0,0,1
    }
    \node at (3.5,-5.5) {$\bfW$};
    \end{shift}

    \draw[->] (19,-2.5) -- (22,-2.5)
    node[midway,above] {$\Rect$};
    
    % DIAGRAM 3
    \begin{shift}[(22,0)]
    \draw[very thick] (0.5,-0.5) -- (6.5,-0.5);
    \draw[very thick] (2.5,-0.5) -- (2.5,-4.5);
    \drawRCGraph{4}{6}{
        9,9,1,1,0,0,
        9,0,0,1,1,0,
        0,0,0,0,0,0,
        1,0,1,0,0,1
    }
    \node at (3.5,-5.5) {$U$};
    
    \begin{shift}[(7,0)]
    \draw[very thick] (0.5,-0.5) -- (1.5,-0.5);
    \draw[very thick] (0.5,-0.5) -- (0.5,-4.5);
    \drawRCGraph{4}{1}{
        2,
        0,
        0,
        0
    }
    \node at (1,-5.5) {\color{y color 1}$P_1^\dagger$};
    \end{shift}
    \end{shift}
\end{tikzpicture}
\caption{
    Illustration of Proposition \ref{prop:stable Groth} with $w = 35241687$.
    Here $V\in \PD^\downarrow(ws_4)$ is paired with $U\in \PD^\downarrow(s_1w)$ since there is a super pipe dream $\bfW\in \SPD(w)$ with $\coRect(\bfW) = (V,P_4)$ and $\Rect(\bfW) = (U,P_1)$.
}
\label{fig:stanley}
\end{figure}
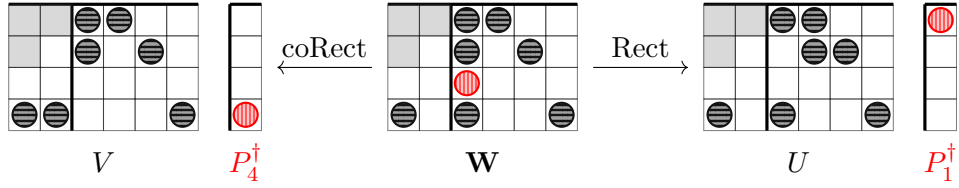

\begin{proposition}
\label{prop:stable Groth}
Let $w\in S_\infty$.
Then \[
    \sum_{w=s_k*v}
    \beta^{\delta_{w,v}} G_v^{(\beta)}(x)
    = \sum_{w=u*s_k}
    \beta^{\delta_{w,u}} G_u^{(\beta)}(x).
\]
\end{proposition}

\begin{proof}
For $k\in\bbN$, write $P_k$ for the pipe dream $\{(1,k)\}\in \PD_0^+(s_k)$.
Given $V\in \PD^\downarrow(v)$ for some $w = s_k*v$, by Propositions \ref{prop:rectification} and \ref{prop:corectification}, the pair $\Rect(\coRect^{-1}(V,P_k))$ will have the form $(U,P_{k'})$ where $U\in \PD(u)$ for some $w = u*s_{k'}$ with $\wt(V) = \wt(U)$.
Mapping $V\mapsto U$ gives a weight-preserving bijection \[
    \coprod_{w=s_k*v} \PD^\downarrow(v)
    \overset\sim\longrightarrow
    \coprod_{w=s_k*u} \PD^\downarrow(u).
    \qedhere
\]
\end{proof}

As usual, taking $\beta = 0$ in Proposition \ref{prop:stable Groth} recovers an identity for Stanley symmetric functions.

\begin{corollary}
\label{cor:Stanley symmetric}
Let $w\in S_\infty$.
Then \[
    \sum_{w\doteq s_kv}
    F_v(x)
    = \sum_{w\doteq us_k}
    F_u(x).
\]
\end{corollary}

\subsection{A recurrence for double Grothendieck polynomials}
\label{subsect:R_k recurrence}

Let $w\in S_\infty$.
Recall that we constructed a bijection for each $k\in\bbZ$ \[
    X^+_{\geq k}\colon \SPD(w)
    \overset\sim\longrightarrow
    \left\{\begin{tabular}{@{} c @{}}
        $\bfQ\in \SPD(w)$ \\
        with no black checkers \\
        in row $k$
    \end{tabular}\right\}
\] with inverse $X^-_{\geq k+1}$ (extended to all of $\SPD(w)$ by sending each $\bfP$ with a black checker in row $k$ to $0$).
Write $R_k$ for the endomorphism of $\bbZ[\beta][x;y]$ corresponding to $X^-_{\geq k+1}$.
Explicitly, $R_k$ acts by the evaluations $x_k\mapsto 0$ and $x_{i+1} \mapsto x_i$ for $i\geq k$.
These are the \textit{Bergeron--Sottile operators} of \cite{BergeronSottile98}.

Recently, Nadeau, Spink, and Tewari used the $R_k$ operators to give an elementary recursion for Schubert polynomials \cite{NadeauSpinkTewari24}.
The following proposition shows that there is a natural extension of this recursion for double Grothendieck polynomials.

\begin{proposition}
\label{prop:R_k Groth recursion}
Let $w\in S_\infty$.
Then \[
    \frakG_w^{(\beta)}(x;y)
    = R_1\big(\frakG_w^{(\beta)}(x;y)\big)
    + \sum_{\substack{
        u,k \\
        w = u*s_k
    }} \beta^{\delta_{w,u}} x_k \cdot R_k\big(\frakG_u^{(\beta)}(x;y)\big).
\]
\end{proposition}

\begin{proof}
For $w\in S_\infty$ and $k\in\bbN$, let $\SPD^+(w,k)$ be temporary notation for the collection of super pipe dreams $\bfQ\in \SPD^+(w)$ with no black checkers in row $k$.
Notice that each $R_k(\frakG_w^{(\beta)}(x;y))$ can be interpreted as the sum of $R_k(\bfQ)$ indexed over $\bfQ\in \SPD^+(w,k)$.

The RHS of the recurrence is thus a weighted enumeration of the set \[
    \SPD^+(w,1)
    \amalg
    \coprod_{w = u*s_k}
    \SPD^+(u,k)
\] where each $\bfQ\in \SPD^+(w,1)$ gets weight $R_1(\wt(\bfQ))$ and each $\bfQ\in \SPD^+(u,k)$ gets weight $\beta^{\delta_{w,u}}x_k\cdot R_k(\wt(\bfQ))$.
We will give a weight-preserving bijection between $\SPD^+(w)$ and the above set.

Let $\bfP\in \SPD^+(w)$.
If $X^+_{\geq 1}\bfP$ ($= X^+\bfP$) is ordinary, then send $\bfP$ to $X^+_{\geq 1}\bfP$.
Notice that $R_1(\wt(X^+_{\geq 1}\bfP)) = \wt(X^-_{\geq 2}X^+_{\geq 1}\bfP) = \wt(\bfP)$.

Otherwise, let $k\in\bbN$ be maximal such that $X^+_{\geq k}\bfP$ is not ordinary.
Since $X^+_{\geq k+1}\bfP$ is ordinary, the reason for $X^+_{\geq k}\bfP$ being extraordinary must be a black checker at position $(k+1,0)$.
By removing this checker from $X^+_{\geq k}\bfP$, we obtain an ordinary super pipe dream $\bfQ\in \SPD^+(u,k)$ for some $u\in S_\infty$ with $w = u*s_k$.
Send $\bfP$ to $\bfQ$.
Notice that $
    \beta^{\delta_{w,u}}x_{k+1}\cdot \wt(\bfQ)
    = \wt(X^+_{\geq k}\bfP)$.
Applying $R_k$ to this equality gives $
    \beta^{\delta_{w,u}}x_k\cdot \wt(\bfQ) 
    = R_k(\wt(X^+_{\geq k}\bfP))
    = \wt(X^-_{\geq k+1}X^+_{\geq k}\bfP)
    = \wt(\bfP)
$.

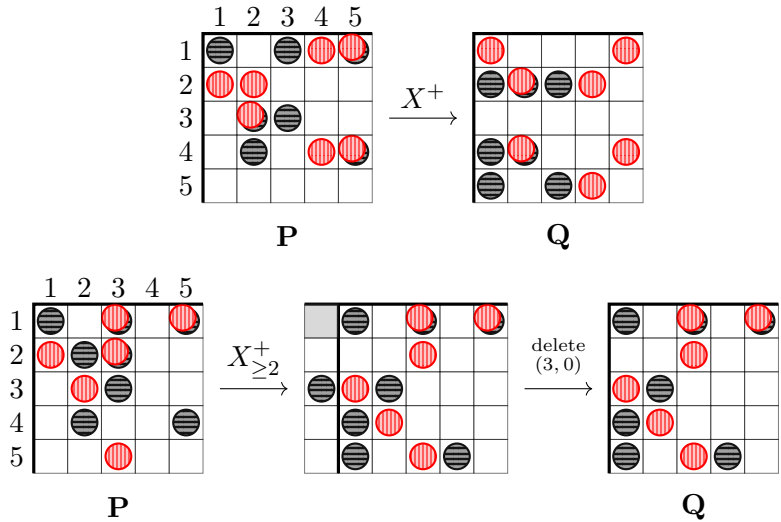
\begin{figure}
\begin{tikzpicture}[scale=0.45]
    % FIRST LINE
    \begin{shift}[(5,0)]
    \draw[very thick] (0.5,-5.5) -- (0.5,-0.5) -- (5.5,-0.5);
    \drawRCGraph{5}{5}{
        1,0,1,2,3,
        2,2,0,0,0,
        0,3,1,0,0,
        0,1,0,2,3,
        0,0,0,0,0
    }
    \drawNorthLabels{0}{1}{1,2,3,4,5}
    \drawWestLabels{1}{0}{1,2,3,4,5}
    \node at (3,-6.5) {$\bfP$};
    \draw[->] (6,-3) -- (8,-3)
    node[midway,above] {$X^+$};
    \end{shift}

    \begin{shift}[(13,0)]
    \draw[very thick] (0.5,-5.5) -- (0.5,-0.5) -- (5.5,-0.5);
    \drawRCGraph{5}{5}{
        2,0,0,0,2,
        1,3,1,2,0,
        0,0,0,0,0,
        1,3,0,0,2,
        1,0,1,2,0
    }
    \node at (3,-6.5) {$\bfQ$};
    \end{shift}

    % SECOND LINE
    \begin{shift}[(0,-8)]
    \draw[very thick] (0.5,-5.5) -- (0.5,-0.5) -- (5.5,-0.5);
    \drawRCGraph{5}{5}{
        1,0,3,0,3,
        2,1,3,0,0,
        0,2,1,0,0,
        0,1,0,0,1,
        0,0,2,0,0
    }
    \drawNorthLabels{0}{1}{1,2,3,4,5}
    \drawWestLabels{1}{0}{1,2,3,4,5}
    \node at (3,-6.5) {$\bfP$};
    \draw[->] (6,-3) -- (8,-3)
    node[midway,above] {$X^+_{\geq2}$};
    \end{shift}

    \begin{shift}[(8,-8)]
    \draw[very thick] (1.5,-0.5) -- (1.5,-5.5);
    \draw[very thick] (0.5,-0.5) -- (6.5,-0.5);
    \drawRCGraph{5}{6}{
        9,1,0,3,0,3,
        0,0,0,2,0,0,
        1,2,1,0,0,0,
        0,1,2,0,0,0,
        0,1,0,2,1,0
    }
    \draw[->] (7,-3) -- (9,-3);
    \node at (8,-1.7) {\tiny delete};
    \node at (8,-2.3) {\tiny $(3,0)$};
    \end{shift}
    
    \begin{shift}[(17,-8)]
    \draw[very thick] (0.5,-5.5) -- (0.5,-0.5) -- (5.5,-0.5);
    \drawRCGraph{5}{5}{
        1,0,3,0,3,
        0,0,2,0,0,
        2,1,0,0,0,
        1,2,0,0,0,
        1,0,2,1,0
    }
    \node at (3,-6.5) {$\bfQ$};
    \end{shift}
\end{tikzpicture}

\caption{
    Illustration of Proposition \ref{prop:R_k Groth recursion} for $w = 261543978$.
    In the top line, $X^+\bfP$ is ordinary, so $\bfP$ is sent to $\bfQ = X^+\bfP$.
    In the bottom line, $X^+_{\geq2}\bfP$ is not ordinary, so $\bfP$ is sent to a $\bfQ\in \SPD^+(ws_2,2)$.
}
\label{fig:R_k recursion}
\end{figure}

It is routine to verify that this assignment is bijective.
\end{proof}

Here is the version of Proposition \ref{prop:R_k Groth recursion} for double Schubert polynomials, which follows either by taking $\beta = 0$ or by restricting the above bijection to $\PD^+_0(w)$:

\begin{corollary}
\label{cor:R_k Schub recursion}
Let $w\in S_\infty$.
Then \[
    \frakS_w(x,y)
    = R_1\left(\frakS_w(x,y)\right)
    + \sum_{
        w \doteq us_k
    } x_k \cdot R_k \left(\frakS_u(x,y)\right).
\]
\end{corollary}

\noindent
\textit{Remark:}\quad Our bijective proof of Corollary \ref{cor:R_k Schub recursion} is essentially the same as the one given in \cite{NadeauSpinkTewari24} for single Schubert polynomials (the specialization $y = 0$).
Precisely, if $\bfP\mapsto \bfQ$ by our bijection and $\Rect(\bfP) = (V,U)$, then $\Rect(\bfQ) = (V',U)$ where $V\mapsto V'$ under their bijection.
\vspace{4pt}

\section{A dual RSK correspondence through pipe dreams}
\label{sect:RSK}

The \textit{dual Robinson--Schensted--Knuth (RSK) correspondence} is a bijection between rectangular binary matrices and pairs of semi-standard Young tableaux of conjugate shape.
Both of these objects can identified with certain classes of pipe dreams: semi-standard Young tableau are known to be equinumerous with reduced pipe dreams for Grassmannian permutations and binary matrices naturally encode super pipe dreams for biGrassmannian permutations.
We thus can obtain a variant of dual RSK directly in terms of pipe dreams though rectification.
In this section, we will examine the relationship between the classical dual RSK correspondence and our pipe dream variant.

All pipe dreams $P$ in this section will be reduced and visualized as wiring diagrams rather than as checker arrangements.
The reduced hypothesis means that no two pipes in $P$ may cross more than once.
When we speak of ``pipe $k$'' in $P$, we are referring to the pipe appearing $k$\textsuperscript{th} from the top in any column strictly west of those containing crosses.
We use the same conventions for super pipe dreams $\bfP$.

\subsection{Grassmannian permutations}
\label{subsect:grass permutations}

We begin by giving the necessary background on Grassmannian permutations and their reduced pipe dreams.

\begin{definition}
Let $m\in\bbN$.
A permutation $w\in S_\infty$ is \textit{($m$-)Grassmannian} if one of the following equivalent conditions hold:
\begin{itemize}
    \item
    $\code(w) = (\lambda_m,\dots,\lambda_1,0,\dots)$ for a partition $\lambda = (\lambda_1\geq \cdots\geq \lambda_m)$.
    \item 
    $\Des(w)$ is contained in $\{m\}$.
\end{itemize}
In this case, we write $w = \Grass(\lambda,m)$ (see Figure \ref{fig:grassmannian pipe dream})
\end{definition}

Let $w = \Grass(\lambda,m)$ be a Grassmannian permutation.
It is well-known that $\frakS_w(x) = s_\lambda(x_1,\dots,x_m)$, the \textit{Schur polynomial} for $\lambda$ in $m$ variables (cf. Proposition \ref{prop:grassmannian}(i)).
We will view $s_\lambda(x_1,\dots,x_m)$ as a generating function for the following family of tableaux:

\begin{definition}
\label{def:rev ssyt}
Let $\lambda$ be a partition.
A \textit{reverse semi-standard Young tableau} (or \textit{rev-tableau}) of shape $\lambda$ is a filling $T\colon Y(\lambda)\to \bbZ$ which is weakly decreasing along rows and strictly decreasing down columns.
Here $Y(\lambda) = \{(i,j) \mid j \leq \lambda_i\}$ is the \textit{diagram} of $\lambda$.
\end{definition}

\noindent
Writing $\RSSYT(\lambda)$ for the set of all rev-tableaux of shape $\lambda$ and $\RSSYT(\lambda,m)$ for the subset of the former taking values in $[m]$, we have \[
    s_\lambda(x_1,\dots,x_m) =
    \sum_{T\in \RSSYT(\lambda,m)} \wt(T)
\] where $\wt(T) := \prod_{(i,j)\in Y(\lambda)} x_{T(i,j)}$.

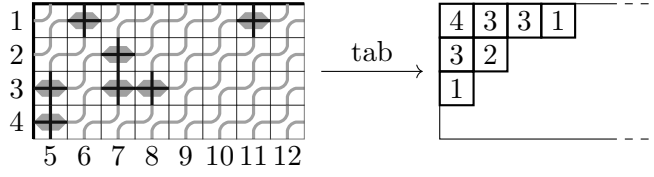
\begin{figure}
\begin{tikzpicture}[scale=0.45]
    % DIAGRAM 1
    \begin{shift}
    \draw[very thick] (0.5,-4.5) -- (0.5,-0.5) -- (8.5,-0.5);
    \drawPD{8}{8}{
        0,1,0,0,0,0,1,0,
        0,0,1,0,0,0,0,0,
        1,0,1,1,0,0,0,0,
        1,0,0,0,0,0,0,0
    }
    \drawNorthLabels{5}{1}{5,6,7,8,9,10,11,12}
    \drawWestLabels{1}{0}{1,2,3,4}

    \draw[->] (9,-2.5) -- (12,-2.5)
    node[midway,above] {$\tab$};
    \end{shift}

    % DIAGRAM 2
    \begin{shift}[(12,0)]
    % \draw[very thick] (0.5,-4.5) -- (0.5,-0.5) -- (5.5,-0.5);
    \foreach \x\y\z in {
    %   1     2     3     4
        1/1/4,1/2/3,1/3/3,1/4/1,
        2/1/3,2/2/2,
        3/1/1
    }{
        \node at (\y,-\x) {\z};
        \draw[thick] (\y-0.5,-\x+0.5) rectangle (\y+0.5,-\x-0.5);
    }
    \draw (5.5,-4.5) -- (0.5,-4.5) -- (0.5,-0.5) -- (5.5,-0.5);
    \draw[dashed] (5.5,-0.5) -- (7,-0.5);
    \draw[dashed] (5.5,-4.5) -- (7,-4.5);
    \end{shift}
\end{tikzpicture}
    
\caption{
    A pipe dream $P\in \PD_0^+(w)$ and the rev-tableau $\tab(P)$ for $w = 13582467 = \Grass(\lambda,4)$ where $\lambda = (4,2,1,0)$.
    The numbers along the border of $P$ indicate the labels of the pipes.
}
\label{fig:grassmannian pipe dream}
\end{figure}

\begin{proposition}
\label{prop:grassmannian}
Let $w = \Grass(\lambda,m)$ be a Grassmannian permutation.
\begin{itemize}
\item[(i)]
There is a weight-preserving bijection \[
    \tab\colon \PD_0(w) \overset\sim\longrightarrow \RSSYT(\lambda)
\] which maps $\PD_0^+(w)$ onto $\RSSYT(\lambda,m)$.
\item[(ii)]
If $P\in \PD_0(w)$ contains a ladder $L$ where a ladder move can be made, then $L$ is a $2$-ladder.
\end{itemize}
\end{proposition}

\begin{proof}
A proof of (ii) can be found in \cite{Gao21}.

For (i), we only provide the construction.
Given $P\in \PD_0(w)$, define $\tab(P)$ to be the rev-tableau whose labels in row $i\in[m]$ are the row indices of the crosses that pipe $m-i+1$ passes (horizontally) through in $P$, both read from left-to-right; see Figure \ref{fig:grassmannian pipe dream} for an example.
By definition, this correspondence is weight preserving.

Equivalently, $\tab(P)$ is the tableau whose labels in column $j\in\bbN$, from top-to-bottom, are the row indices of the crosses that pipe $m+j$ passes (vertically) through in $P$ read from bottom-to-top.
\end{proof}

For $m,n\geq 1$, there is a unique $m$-Grassmannian permutation whose inverse is $n$-Grassmannian, namely $b(m,n) := \Grass((n,\dots,n),m)$.
Such permutations are called \textit{biGrassmannian}.

We close this section by collecting some results on biGrassmannian permutations.
Recall that the \textit{conjugate} (or \textit{transpose}) of a partition $\lambda$ is the partition $\lambda^t$ determined by $Y(\lambda^t) = Y(\lambda)^t$.

\begin{proposition}
\label{prop:bigrassmannian}
Let $b(m,n)$ be a biGrassmannian permutation.
\begin{itemize}
\item[(i)]
$\PD_0^+(b(m,n))$ consists of a single pipe dream, namely $[m]\times [n]$.
\item[(ii)]
If $b(m,n) \doteq u^{-1}v$, then there is a partition $\lambda\sseq [m]\times [n]$ such that $v = \Grass(\lambda,m)$ and $u = \Grass(\lambda^\dagger,n)$ where\footnote{
    Notice that we suppress $m$ and $n$ in the notation for $\lambda^\dagger$.
} \begin{align*}
    \lambda^\dagger 
    := (n-\lambda_m+1,\dots,n-\lambda_1+1)^t
    = (m-\lambda^t_n+1,\dots,m-\lambda^t_1+1).
\end{align*}
\end{itemize}
\end{proposition}

\subsection{Insertion for reduced Grassmannian pipe dreams}

Suppose $P\in \PD_0(w)$ for some $w = \Grass(\lambda,m)$, and let $I\sseq [m]$.
Using the insertion algorithm from Section \ref{subsect:pieri rule}, we can form a new pipe dream $I\msert P\in \PD^+(v)$ for some $v\in S_\infty$.
By the discussion after Definition \ref{def:insertion}, $v$ is also $m$-Grassmannian, so we can consider the rev-tableau $\tab(I\msert P)$.

\begin{figure}
\begin{tikzpicture}[scale=0.4]
    % TOP DIAGRAM
    \begin{shift}[(4,7)]
    \draw[very thick] (0.5,-4.5) -- (0.5,-0.5) -- (8.5,-0.5);
    \drawPD{4}{8}{
        2,0,1,0,0,0,0,1,
        2,0,0,1,0,0,0,0,
        1,1,0,1,1,0,0,0,
        2,1,0,0,0,0,0,0
    }
    \end{shift}

    % TOP-LEFT DIAGRAM
    \begin{shift}[(0,0)]
    % super pipe dream
    \draw[very thick] (0.5,-0.5) -- (12.5,-0.5);
    \draw[very thick] (4.5,-0.5) -- (4.5,-5.5);
    \drawPD{4}{12}{
        9,9,9,8,0,0,1,0,0,0,0,1,
        9,9,8,0,0,0,0,1,0,0,0,0,
        9,8,1,0,0,1,0,1,1,0,0,0,
        8,0,0,0,0,1,0,0,0,0,0,0
    }
    \begin{shift}[(0,-4)]\drawPD{1}{5}{0,0,0,0,0}\end{shift}
    \begin{shift}[(0,-5)]\drawPD{1}{4}{2,0,0,0}\end{shift}
    \begin{shift}[(0,-6)]\drawPD{1}{3}{2,0,0}\end{shift}
    \begin{shift}[(0,-7)]\drawPD{1}{2}{2,0}\end{shift}
    \begin{shift}[(0,-8)]\drawPD{1}{1}{0}\end{shift}

    % tableau
    \begin{shift}[(7,-5)]
    % - red cells
    \foreach \x\y\z in {1/1/8,2/1/7,3/1/6}{
        \fill[color=y color 2] (\y-0.5,-\x+0.5) rectangle (\y+0.5,-\x-0.5);
        \node at (\y,-\x) {\color{y color 1}\z};
        \draw[thick] (\y-0.5,-\x+0.5) rectangle (\y+0.5,-\x-0.5);
    }
    % - black cells
    \foreach \x\y\z in {
              1/2/4,1/3/3,1/4/3,1/5/1,
              2/2/3,2/3/2,
              3/2/1,
        4/1/3
    }{
        \node at (\y,-\x) {\z};
        \draw[thick] (\y-0.5,-\x+0.5) rectangle (\y+0.5,-\x-0.5);
    }
    \end{shift}
    
    \node (A) at (14,-2.5) {$\bfW_0$};
    \node (B) at (15,4.5) {$\bfW = \bfW_4$};
    \node (C) at (20,-2.5) {$\bfW_9$};

    \draw[->] (B) -- (C) node[midway,above right] {\scalebox{0.7}{$(Y^+)^5$}};
    \draw[->] (A) -- (C) node[midway,below] {\scalebox{0.7}{$(Y^+)^9$}};;
    
    \draw[->] (6.5,0) -- (6.5,2) node[midway,left] {\scalebox{0.7}{$(Y^+)^4$}};
    
    \draw[->] (6.5,-6) -- (6.5,-12) node[near end,left] {\scalebox{0.7}{$Y'_5\cdots Y'_{-3}$}};
    \end{shift}

    % BOTTOM-LEFT DIAGRAM
    \begin{shift}[(0,-17)]
    % super pipe dream
    \draw[very thick] (0.5,-0.5) -- (12.5,-0.5);
    \draw[very thick] (4.5,4.5) -- (4.5,-5.5);
    \begin{shift}[(4,5)]\drawPD{5}{8}{
                9,9,9,9,8,2,0,0,
                9,9,9,8,0,0,0,0,
                9,9,8,0,0,0,0,0,
                9,8,0,0,0,0,0,0,
                8,0,0,0,0,0,0,0
    }\end{shift}
    \drawPD{4}{12}{
        9,9,9,8,0,0,1,0,0,0,0,1,
        9,9,8,0,0,0,0,1,0,0,0,0,
        9,8,0,1,0,1,0,1,1,0,0,0,
        8,0,0,0,0,1,0,0,0,0,0,0
    }
    \begin{shift}[(0,-4)]\drawPD{1}{5}{0,0,0,0,0}\end{shift}
    \begin{shift}[(0,-5)]\drawPD{1}{4}{0,0,0,0}\end{shift}
    \begin{shift}[(0,-6)]\drawPD{1}{3}{2,0,0}\end{shift}
    \begin{shift}[(0,-7)]\drawPD{1}{2}{2,0}\end{shift}
    \begin{shift}[(0,-8)]\drawPD{1}{1}{0}\end{shift}

    % tableau
    \begin{shift}[(7,-5)]
    % - red cells
    \foreach \x\y\z in {1/1/8,2/1/7,4/1/-4}{
        \fill[color=y color 2] (\y-0.5,-\x+0.5) rectangle (\y+0.5,-\x-0.5);
        \node at (\y,-\x) {\color{y color 1}\z};
        \draw[thick] (\y-0.5,-\x+0.5) rectangle (\y+0.5,-\x-0.5);
    }
    % - black cells
    \foreach \x\y\z in {
              1/2/4,1/3/3,1/4/3,1/5/1,
              2/2/3,2/3/2,
        3/1/3,3/2/1
    }{
        \node at (\y,-\x) {\z};
        \draw[thick] (\y-0.5,-\x+0.5) rectangle (\y+0.5,-\x-0.5);
    }
    \end{shift}
    
    \draw[->] (13,-2.5) -- (17,-2.5) node[midway,above] {\scalebox{0.7}{$Y'_5\cdots Y'_{-3}$}};
    \end{shift}

    % BOTTOM-RIGHT DIAGRAM
    \begin{shift}[(17,-17)]
    % super pipe dream
    \draw[very thick] (0.5,-0.5) -- (12.5,-0.5);
    \draw[very thick] (4.5,4.5) -- (4.5,-5.5);
    \begin{shift}[(4,5)]\drawPD{5}{8}{
                9,9,9,9,8,2,0,0,
                9,9,9,8,0,0,0,0,
                9,9,8,0,0,2,0,0,
                9,8,0,0,0,0,0,0,
                8,0,0,0,0,0,0,0
    }\end{shift}
    \drawPD{4}{12}{
        9,9,9,8,0,1,0,0,0,0,0,1,
        9,9,8,0,0,0,0,1,0,0,0,0,
        9,8,0,0,1,1,0,1,1,0,0,0,
        8,0,0,0,0,1,0,0,0,0,0,0
    }
    \begin{shift}[(0,-4)]\drawPD{1}{5}{0,0,0,0,0}\end{shift}
    \begin{shift}[(0,-5)]\drawPD{1}{4}{0,0,0,0}\end{shift}
    \begin{shift}[(0,-6)]\drawPD{1}{3}{0,0,0}\end{shift}
    \begin{shift}[(0,-7)]\drawPD{1}{2}{2,0}\end{shift}
    \begin{shift}[(0,-8)]\drawPD{1}{1}{0}\end{shift}

    % tableau
    \begin{shift}[(7,-5)]
    % - red cells
    \foreach \x\y\z in {1/1/8,3/2/-2,4/1/-4}{
        \fill[color=y color 2] (\y-0.5,-\x+0.5) rectangle (\y+0.5,-\x-0.5);
        \node at (\y,-\x) {\color{y color 1}\z};
        \draw[thick] (\y-0.5,-\x+0.5) rectangle (\y+0.5,-\x-0.5);
    }
    % - black cells
    \foreach \x\y\z in {
              1/2/4,1/3/3,1/4/3,1/5/1,
        2/1/3,2/2/3,2/3/2,
        3/1/1
    }{
        \node at (\y,-\x) {\z};
        \draw[thick] (\y-0.5,-\x+0.5) rectangle (\y+0.5,-\x-0.5);
    }
    \end{shift}
    \end{shift}

    % TOP-RIGHT DIAGRAM
    \begin{shift}[(17,0)]
    % super pipe dream
    \draw[very thick] (4.5,-0.5) -- (12.5,-0.5);
    \draw[very thick] (4.5,4.5) -- (4.5,-4.5);
    \begin{shift}[(4,5)]
    \drawPD{9}{8}{
        9,9,9,9,8,2,0,0,
        9,9,9,8,0,0,0,0,
        9,9,8,0,0,2,0,0,
        9,8,0,0,0,0,0,0,
        8,0,0,0,0,2,0,0,
        0,1,0,0,0,0,0,1,
        0,0,1,0,0,0,0,0,
        1,0,1,1,1,0,0,0,
        1,0,0,0,0,0,0,0
    }
    \end{shift}

    % tableau
    \begin{shift}[(7,-5)]
    % - red cells
    \foreach \x\y\z in {2/3/0,3/2/-2,4/1/-4}{
        \fill[color=y color 2] (\y-0.5,-\x+0.5) rectangle (\y+0.5,-\x-0.5);
        \node at (\y,-\x) {\color{y color 1}\z};
        \draw[thick] (\y-0.5,-\x+0.5) rectangle (\y+0.5,-\x-0.5);
    }
    % - black cells
    \foreach \x\y\z in {
        1/1/4,1/2/3,1/3/3,1/4/3,1/5/1,
        2/1/3,2/2/2,
        3/1/1
    }{
        \node at (\y,-\x) {\z};
        \draw[thick] (\y-0.5,-\x+0.5) rectangle (\y+0.5,-\x-0.5);
    }
    \end{shift}
    
    \draw[<-] (6.5,-6) -- (6.5,-12) node[near end,right] {\scalebox{0.7}{$Y'_5\cdots Y'_{-3}$}};
    \end{shift}
     
\end{tikzpicture}

\caption{
    Insertion of $I = \{3\}$ into the pipe dream $P$ from Figure \ref{fig:grassmannian pipe dream}.
    The output $3\overset4\to P$ consists of the black crosses in $\bfW_9 = (Y^+)^9\bfW_0 = (Y^+)^5\bfW$, where $\bfW\in \SPD^+(w\ominus_4 1)$ corresponds to the triple $(P,I,I)$ as in the proof of Proposition \ref{prop:K-pieri rule}.
    The bottom route shows $\bfW_9$ computed as $(Y'_5\cdots Y'_{-3})^3\bfW_0$ with the tableau at each step shown.
}
\label{fig:grassmannian insertion}
\end{figure}
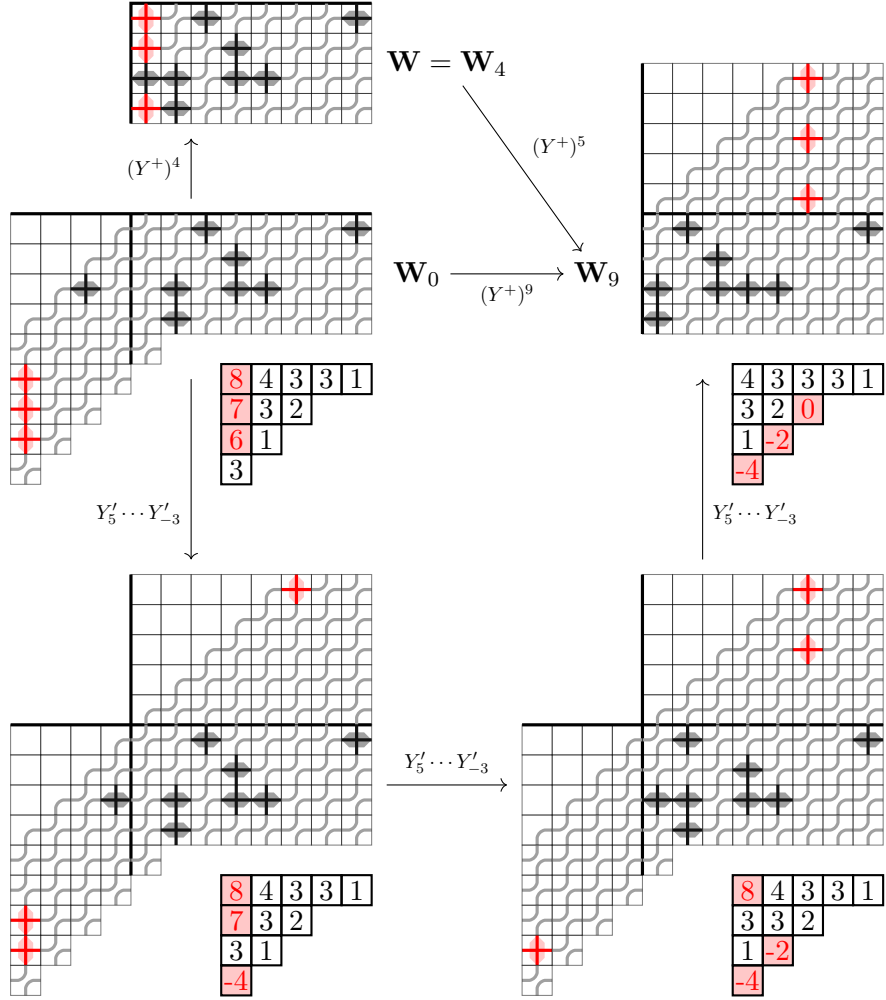

The goal of this section is to show that $\tab(I\msert P)$ can be obtained as the product\footnote{
    The product as SSYT in the alphabet $\bbZ$ with the opposite ordering.
} $I*\tab(P)$ in the tableaux monoid (Proposition \ref{prop:insertion agrees}), where we identify $I$ with the column tableau filled with the elements of $I$.
We defer the precise definition of the product and other operations on Young tableaux (e.g. jeu de taquin slides) to \cite{Fulton97}.
The key is the following lemma:

\begin{lemma}
\label{lem:grassmannian insertion}
Let $w$ be a Grassmannian permutation and fix $j\in\bbZ$.
Suppose $\bfP\in \SPD_0(w)$ has at least $k$ red crosses in column $j$.

\begin{itemize}
    \item[(i)]
    If $\bfP\in \SPD_0(w)$ has no red crosses in column $j+1$, then for all $k\geq 1$, \[
        Y'_j(Y'_{j+1})^{k-1}(Y'_j)^{k-1}\bfP = (Y'_{j+1})^{k-1}(Y'_j)^k\bfP.
    \]
    
    \item[(ii)]
    If $\bfP\in \SPD_0(w)$ has no red crosses in columns $j+1,\dots,j+n$ for some $n\geq 1$, then for all $k\geq 1$, \[
        (Y'_{j+n-1})^k\cdots (Y'_j)^k\bfP
        = (Y'_{j+n-1}\cdots Y'_j)^k\bfP.
    \]
\end{itemize}
\end{lemma}

\begin{proof}
Without loss of generality, we may assume that $j = 1$ for both statements.

We first show (i), so assume $\bfP$ has no red crosses in column $2$.
Write $i_1< i_2< \cdots$ for the rows with red crosses in column $1$ of $\bfP$.
For each $\ell\in [k-1]$, write $L_\ell$ for the ladder formed when applying $Y'_1$ to $(Y'_1)^{\ell-1}\bfP$.
By Proposition \ref{prop:grassmannian}(ii), $L_\ell$ can only contain crosses in column $1$ of $(Y'_1)^{\ell-1}\bfP$.
In particular, $(i_\ell,1)$ is the only red cross in $(Y'_1)^{\ell-1}\bfP$ appearing in $L_\ell$.
Hence $i'_1< \cdots < i'_k$ where $i'_\ell$ denotes the top row of $L_\ell$.

Now, applying $(Y'_2)^{k-1}$ to $(Y'_1)^{k-1}\bfP$ would only alter rows at or above row $i'_{k-1}$.
On the other hand, $Y'_k$ will only affect $(Y'_1)^{k-1}\bfP$ in rows between (and including) $i'_k$ and $i_k$.
Since $i'_{k-1} < i'_k\leq i_k$, we conclude that $(Y'_2)^{k-1}$ and $Y'_1$ can be applied to $(Y'_1)^{k-1}\bfP$ in either order to obtain the same super pipe dream.

We prove (ii) by induction on $k$ (the $k=1$ case being trivial).
The induction hypothesis states \[
    (Y'_n)^{k-1}\cdots (Y'_1)^{k-1}\bfP
    = (Y'_n\cdots Y'_1)^{k-1}\bfP.
\]
Applying $Y'_n\cdots Y'_1$ to both sides and observing that $Y'_pY'_q = Y'_qY'_p$ whenever $|p-q|\geq 2$, we get \[
    Y'_n
    \left(Y'_{n-1}(Y'_n)^{k-1}\right)\cdots
    \left(Y'_1(Y'_2)^{k-1}\right)
    (Y'_1)^{k-1}\bfP
    = (Y'_n\cdots Y'_1)^k\bfP.
\]
Now apply the previous part $n-1$ times to the above to get \[
    (Y'_n)^k\cdots (Y'_1)^k\bfP
    = (Y'_n\cdots Y'_1)^k\bfP.
    \qedhere
\]
\end{proof}

Given a super pipe dream $\bfP\in \SPD_0(w)$ for $w$ Grassmannian, write $\tab(\bfP)$ for the rev-tableau of the underlying pipe dream $\tab(P)$ with each label colored to match the cross that it came from.
The next proposition gives a description of how the operator $Y'_j$ is reflected in the rev-tableau $\tab(\bfP)$. 

\begin{proposition}
\label{prop:grass pipe dream jdt}
Let $w = \Grass(\lambda,m)$ be a Grassmannian permutation and suppose $\bfP\in \SPD_0(w)$ has a red cross in some column $j\geq 1$.
Write $(a,b)$ for the cell in $\tab(\bfP)$ corresponding to the highest red cross $(i,j)$ in column $j$ of $\bfP$ and $\bfQ := Y'_j\bfP$.

Then $\tab(\bfQ)$ is obtained from $\tab(\bfP)$ by preforming one of the following:

\begin{itemize}
\item[$\lozenge$]
Suppose the cell at $(a,b+1)$ has label less than $i$ (or does not exist).
Let $k\geq 0$ be maximum such that the cell at $(a+k,b)$ has label $i-k$.
Delete $(a,b)$ from $\tab(\bfP)$ creating a hole, slide the cells at $(a+1,b),\dots,(a+k,b)$ up by one, then fill in the hole at $(a+k,b)$ with a new red cell labeled $i-k-1$.

\item[$\blacklozenge$]
Suppose the cell at $(a,b+1)$ has label $i$.
Delete $(a,b)$ from $\tab(\bfP)$ creating a hole, slide the cell at $(a,b+1)$ to the left by one, then fill in the hole at $(a,b+1)$ with a new red cell labeled $i$.
\end{itemize}
\end{proposition}

\begin{proof}
Let $L$ denote the ladder formed when applying $Y'_j$ to $\bfP$.
By construction, the horizontal and vertical pipes passing through $(i,j)$ in $\bfP$ are pipes $m-a+1$ and $m+b$.
It follows that the cell at $(a,b+1)$ in $\tab(\bfP)$ has label $i$ if and only if $(i,j+1)$ is a cross in $\bfP$ (since this is the only situation where pipe $m-a+1$ passes through another cross east of $(i,j)$ in row $i$) if and only if $L$ is a big ladder.

Suppose the cell at $(a,b+1)$ in $\tab(\bfP)$ has label $i$ so that $L$ is a big ladder.
As in the proof of Proposition \ref{lem:grassmannian insertion}, the crosses of $\bfP$ in $L$ are $(i,j),(i-1,j),\dots,(i-k,j)$ with the only red cross being $(i,j)$ (here $i-k-1$ is the top row of $L$).
Pipe $b$ passes vertically through each of these crosses, so the corresponding cells in $\tab(\bfP)$ must be $(a,b),(a+1,b),\dots,(a+k-1,b)$ with labels $i,i-1,\dots,i-k$, respectively.
Since $(i-k-1,j)$ is not a cross in $\bfP$, the cell $(a+k+1,b)$ must either not exist or have label less than $i-k-1$.

In $\bfQ$, $L$ consists of the crosses $(i-1,j+1),(i-2,j+1),\dots,(i-k+1,j+1)$ with the only red cross being $(i-k+1,j+1)$.
For each $r\in [k+1]$, the same pipes passing through $(i-r+1,j)$ in $\bfP$ also pass through $(i-r,j+1)$ in $\bfQ$, so the boxes corresponding to the crosses in $L$ in $\tab(\bfQ)$ are the same as those in $\tab(\bfP)$.
The labels here are $i-1,i-2,\dots,i-k+1$ with $i-k+1$ red and the others black, hence $\tab(\bfQ)$ matches the description given in the proposition.

Now, suppose that $L$ is a $1$-ladder.
Since pipe $a$ passes horizontally through $L$, the cells in $\tab(\bfP)$ corresponding to the crosses in $L$ are $(a,b)$ and $(a,b+1)$.
The labels of these boxes are both $i$ with the one on the left being red.

$\bfQ$ is formed from $\bfP$ by exchanging the red cross at $(i,j)$ with the black cross at $(i,j+1)$.
Since the underlying pipe dream does not change, $L$ corresponds to the same boxes in $\tab(\bfQ)$.
The labels of these boxes are again both $i$, but their colors are swapped from how they are in $\tab(\bfP)$.
We conclude that $\tab(\bfQ)$ has the desired form.
\end{proof}

\begin{proposition}
\label{prop:insertion agrees}
Let $w = w(\lambda,m)$ be a Grassmannian permutation, $P\in \PD_0^+(w)$, and $I\sseq [m]$.
Then \[
    \tab(I\msert P) = I * \tab(P).
\]
\end{proposition}

\begin{proof}
Form the super pipe dream $\bfW\in \SPD_0^+(w\ominus_m 1)$ corresponding to the triple $(P,I,I)$ as in the proof of Proposition \ref{prop:K-pieri rule}.
Notice that $w\ominus_m 1 = \Grass(\lambda',m)$ where $\lambda' = (\lambda_1+1,\dots,\lambda_m+1)$.

Set $\bfW_0 := (Y^-)^m\bfW$.
The only change when passing from $\bfW$ to $\bfW_0$ is that the crosses in column $1$ get spread across columns $1-m,\dots,1$, with all the red crosses moved to column $1-m$ (strictly below row $m$).
Since pipe $m+1$ passes through these crosses in $\bfW$, it must also pass though their counterparts in $\bfW_0$.

The result is that $\tab(\bfW_0)$ agrees with $\tab(\bfW)$ everywhere except for in column $1$, where in this column the black cells in $\bfW$ have sunk to the bottom in $\bfW_0$.
From this description, letting $T_0$ denote the skew tableau of black cells in $\tab(\bfW_0)$, we have $\Rect(T_0) = I*\tab(P)$ where the former denotes the usual rectification of tableau.

Take $M\geq 0$ large enough so that the red crosses of $(Y^+)^M\bfW$ all lie in rows at or above row $0$.
We can compute

\begin{align*}
    (Y^+)^M \bfW
    &= (Y^+)^{M+m} \bfW_0 \\
    &= Y^+_MY^+_{M-1}\cdots Y^+_{1-m} \bfW_0 \\
    &= (Y'_M)^k(Y'_{M-1})^k\dots (Y'_{1-m})^k \bfW_0 \tag{cf. Section \ref{subsect:reduced flowing}} \\
    &= (Y'_M Y'_{M-1} \cdots Y'_{1-m})^k \bfW_0
    \tag{Lemma \ref{lem:grassmannian insertion}}
\end{align*}
where $k = m - \abs{I}$ is the number of red crosses.

For $i\in [k]$, set $\bfW_i := (Y'_M Y'_{m-1} \cdots Y'_{1-m})^i \bfW_0$ and write $T_i$ for the subdiagram of $\tab(\bfW_i)$ consisting of the cells with black labels.
Since the red crosses in $\bfW_k = (Y^+)^M\bfW$ appear in rows strictly above the black crosses, $T_k = \tab(I\msert P)$ since $I\msert P$ is by definition the first coordinate of $\Rect(\bfW)$.

On the other hand, when passing from $\bfW_{i-1}$ to $\bfW_i$, the highest red cross in column $1-m$ is flowed all the way to column $M+1$.
The corresponding cell at $(k-i+1,1)$ in $\tab(\bfW_{i-1})$ is moved to a position in $\tab(\bfW_i)$ where the only cells further southeast are red (these are the only cells with labels $\leq 0$).
By the description given in Proposition \ref{prop:grass pipe dream jdt}, $T_i$ is obtained from $T_{i-1}$ by performing a series of jeu de taquin slides to the inner corner at $(k-i+1,1)$, moving it to an outer corner.
We conclude that $T_k$ is the rectification of $T_0$, therefore $\tab(I\msert P) = T_k = \Rect(T_0) = I * \tab(P)$.
\end{proof}

\subsection{A dual Robinson--Schensted--Knuth correspondence}

Fix $m,n\geq 1$.
By Proposition \ref{prop:bigrassmannian}(ii), the rectification bijection realizing the Cauchy identity for $\frakS_{b(m,n)}(x;y)$ (Theorem \ref{thm:cauchy identity}) has the form \[
    \Rect\colon \SPD_0^+(b(m,n))
    \overset\sim\longrightarrow
    \coprod_{\lambda\sseq [m]\times[n]} \PD_0^+(\Grass(\lambda,m))\times \PD_0^+(\Grass(\lambda^\dagger,n)).
\]

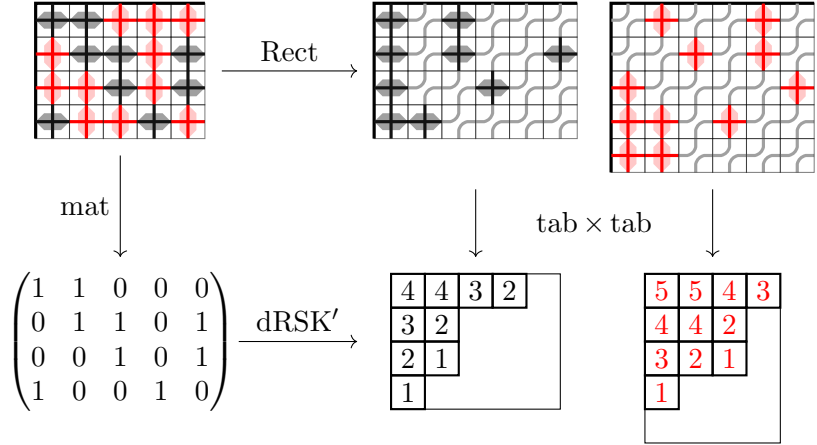
\begin{figure}
\begin{tikzpicture}[scale=0.45]
    % TOP-LEFT
    \begin{shift}
    \draw[very thick] (0.5,-4.5) -- (0.5,-0.5) -- (5.5,-0.5);
    \drawPD{4}{5}{
        1,1,2,2,2,
        2,1,1,2,1,
        2,2,1,2,1,
        1,2,2,1,2
    }
    \end{shift}

    \draw[->] (6,-2.5) -- (10,-2.5) node[midway,above] {$\Rect$};
    \draw[->] (3,-5) -- (3,-8) node[midway,left] {$\mat$};
    
    % TOP-RIGHT
    \begin{shift}[(10,0)]
    \draw[very thick] (0.5,-4.5) -- (0.5,-0.5) -- (6.5,-0.5);
    \drawPD{4}{6}{
        1,0,1,0,0,0,
        1,0,1,0,0,1,
        1,0,0,1,0,0,
        1,1,0,0,0,0
    }
    \draw[->] (3.5,-6) -- (3.5,-8);
    \node at (7,-7) {$\tab\times \tab$};
    \end{shift}
    
    \begin{shift}[(17,0)]
    \draw[very thick] (0.5,-5.5) -- (0.5,-0.5) -- (6.5,-0.5);
    \drawPD{5}{6}{
        0,2,0,0,2,0,
        0,0,2,0,2,0,
        2,0,0,0,0,2,
        2,2,0,2,0,0,
        2,2,0,0,0,0
    }
    \draw[->] (3.5,-6) -- (3.5,-8);
    \end{shift}
    
    % BOTTOM-LEFT
    \begin{shift}[(0,-8)]
    \node at (3,-2.5) {$\begin{pmatrix}
        1 & 1 & 0 & 0 & 0 \\
        0 & 1 & 1 & 0 & 1 \\
        0 & 0 & 1 & 0 & 1 \\
        1 & 0 & 0 & 1 & 0
    \end{pmatrix}$};
    
    \draw[->] (6.5,-2.5) -- (10,-2.5) node[midway,above] {$\RSK'$};
    \end{shift}
    
    % BOTTOM-RIGHT
    \begin{shift}[(10.5,-8)]
    \foreach \x\y\z in {
        1/1/4,1/2/4,1/3/3,1/4/2,
        2/1/3,2/2/2,
        3/1/2,3/2/1,
        4/1/1
    }{
        \node at (\y,-\x) {\z};
        \draw[thick] (\y-0.5,-\x+0.5) rectangle (\y+0.5,-\x-0.5);
    }
    \draw (0.5,-0.5) rectangle ++(5,-4);
    \end{shift}
    
    \begin{shift}[(18,-8)]
    \foreach \x\y\z in {
        1/1/5,1/2/5,1/3/4,1/4/3,
        2/1/4,2/2/4,2/3/2,
        3/1/3,3/2/2,3/3/1,
        4/1/1
    }{
        % \fill[color=y color 2] (\y-0.5,-\x+0.5) rectangle (\y+0.5,-\x-0.5);
        \node at (\y,-\x) {\color{y color 1}\z};
        \draw[thick] (\y-0.5,-\x+0.5) rectangle (\y+0.5,-\x-0.5);
    }
    \draw (0.5,-0.5) rectangle ++(4,-5);
    \end{shift}
\end{tikzpicture}
    
\caption{
    Computation of $\RSK'(A)$ for a binary matrix $A\in \BM_{4\times 5}$.
}

% xvec = [(1,1),(1,2),(2,2),(2,3),(2,5),(3,3),(3,5),(4,1),(4,4)]
% yvec = [(1,3),(1,4),(1,5),(2,1),(2,4),(3,1),(3,2),(3,4),(4,2),(4,3),(4,5)]
\label{fig:pipe dream dual RSK}
\end{figure}

Let $\BM_{m\times n}$ denote the collection of $m$-by-$n$ \textit{binary matrices}---matrices with entries in $\{0,1\}$.
There is a natural identification $\mat\colon \SPD_0^+(b(m,n))\to \BM_{m\times n}$ given by sending a super pipe dream $\bfW\in \SPD_0^+(b(m,n))$ (which by Proposition \ref{prop:bigrassmannian}(i) is the data of a decomposition $[m]\times [n] = W_x\amalg W_y$) to the binary matrix $A$ whose $1$'s and $0$'s are determined by $W_x$ and $W_y$, respectively.
We thus get a commutative square of bijections \[
\begin{tikzcd}
    \SPD_0^+(b(m,n))
    \arrow[r,"\Rect"] \arrow[d, swap, "{\mat}"]
    &
    \displaystyle\coprod_{\lambda\sseq [m]\times[n]} \PD_0^+(\Grass(\lambda,m))\times \PD_0^+(\Grass(\lambda^\dagger,n))
    \arrow[d, "{\tab\times\tab}"]
    \\
    \BM_{m\times n}
    \arrow[r,"\RSK'"]
    &
    \displaystyle\coprod_{\lambda\sseq [m]\times[n]}
    \RSSYT(\lambda,m) \times \RSSYT(\lambda^\dagger,n)
    \end{tikzcd}.
\]
where the bottom map can be thought of as an incarnation of the dual Robinson--Schensted--Knuth correspondence.

We would like to understand $\RSK'$ in terms of tableau insertion.
Let $\bfW\in \SPD_0^+(b(m,n))$ be with corresponding binary matrix $A= \mat(\bfW)\in \BM_{m\times n}$.
Write $(V,U) = \Rect(\bfW)$ and $C_j = \{i\in [m] \mid A_{ij} = 1\}$.
It is not hard to verify that \[
    V = C_1\msert (C_2\msert(\cdots (C_n\msert \varnothing)\cdots))
\] by utilizing that $Y^+_p$ and $Y^+_q$ commute whenever $\abs{p-q}\geq 2$.
By Proposition \ref{prop:insertion agrees}, we conclude that $\tab(V)$ is the usual \textit{insertion tableau} $\ins(A) = C_1 * \cdots * C_n$.

To interpret $\tab(U)$, we use the symmetry theorem for rectification (Proposition \ref{prop:symmetry theorem}) to write $(U,V) = \Rect(\bfW^\dagger)$.
Then by the previous case, $\tab(U) = \ins(A^\dagger)$ where $A^\dagger$ denotes the matrix obtained from $A^t$ by swapping the $0$'s and $1$'s.
Equivalently, $\tab(U) = R_1 * \cdots * R_m$ where $R_i = \{j\in[n] \mid A_{ij} = 0\}$.
We summarize this in the following proposition:
\begin{proposition}
\label{prop:pipe dream dual RSK}
Let $A\in \BM_{m\times n}$.
Then \[
    \RSK'(A) = (\ins(A),\ins(A^\dagger)).
\]
\end{proposition}

To close, we remark that the usual dual RSK correspondence is the bijection \[
    \RSK\colon
    \BM_{m\times n}
    \overset\sim\longrightarrow
    \coprod_{\lambda\sseq [m]\times[n]}
    \RSSYT(\lambda,m) \times \RSSYT(\lambda^t,n).
\] where we send a matrix $A$ to the pair $(\ins(A),\rec(A))$.
Here $\rec(A)$ is the \textit{(dual) recording tableau} of $A$ whose label in cell $(i,j)$ is the largest $r\in [n]$ such that cell $(j,i)$ appears in $C_r * \cdots * C_n$.

There are natural identifications $\RSSYT(\lambda^t,n) \overset\sim\to \RSSYT(\lambda^\dagger,n)$ where a tableau $T$ is sent to the tableau $\ol{T}$ whose labels in column $j\in [m]$ are the values in $[n]$ not appearing as labels in column $m+1-j$ of $T$.\footnote{
    Again, we suppress $m$ and $n$ in the notation for $\ol{T}$. 
}
Computation suggests that the following symmetry of dual RSK holds:
\begin{conjecture}
\label{conj:pipe dream dual RSK is dual RSK}
    Let $A\in \BM_{m\times n}$.
    Then \[
        \ol{\rec(A)} = \ins(A^\dagger).
    \]
\end{conjecture}
\noindent
Given the conjecture, our variant of dual RSK would genuinely become dual RSK through the commutative triangle \[
    \begin{tikzpicture}
    \node at (0,0) {
        $\BM_{m\times n}$
    };
    \draw[->] (0.5,0.5) -- (2.5,1)
    node[midway,shift={(0,0.4)}] {
        \scalebox{0.8}{$\RSK$}
    };
    \draw[->] (0.5,-0.5) -- (2.5,-1) 
    node[midway,shift={(0,-0.4)}] {
        \scalebox{0.8}{$\RSK'$}
    };
    
    \node at (6,1) {
        $\displaystyle\coprod_{\lambda\sseq [m]\times[n]}
        \RSSYT(\lambda,m) \times \RSSYT(\lambda^t,n)$
    };
    \node at (6,-1){
        $\displaystyle\coprod_{\lambda\sseq [m]\times[n]}
        \RSSYT(\lambda,m) \times \RSSYT(\lambda^\dagger,n)$
    };
    \draw[->] (6,0.5) -- (6,-0.5)
    node[midway,shift={(1.2,0)}] {
        \scalebox{0.8}{$(T,S)\mapsto (T,\ol{S})$}
    };
    \end{tikzpicture}.
\]

\section*{Acknowledgment}
The author would like to thank David Anderson and Zachary Hamaker, who have each given valuable mentorship and suggestions on where to take this project.
Additional thanks is given to William Newman, Oliver Pechenik, and Anna Weigandt for helpful conversations.

\nocite{*}
\bibliographystyle{plain}
\bibliography{main}

\end{document}